\documentclass[11pt]{amsart}

\usepackage[
  margin=1in
]{geometry}

\usepackage{amscd}
\usepackage{amssymb}
\usepackage{amsmath}
\usepackage{amsfonts}
\usepackage{enumerate}
\usepackage[all,cmtip]{xy}
\usepackage{tikz-cd}
\usepackage{hyperref}
\usepackage{xcolor}
\usepackage{adjustbox}
\usepackage{graphicx}
\usepackage[T1]{fontenc}
\usepackage{lmodern}

\newtheorem{thm}{Theorem}[section]
\newtheorem{prop}[thm]{Proposition}
\newtheorem{lem}[thm]{Lemma}
\newtheorem{cor}[thm]{Corollary}

\theoremstyle{definition}
\newtheorem{Def}[thm]{Definition}
\newtheorem{exmp}[thm]{Example}

\newtheorem{rem}[thm]{Remark}

\newcommand{\Ss}{\mathbb{S}}
\newcommand{\Zz}{\mathbb{Z}}
\newcommand{\Ee}{\mathbb{E}}
\newcommand{\Ff}{\mathbb{F}_p}

\newcommand{\BP}[1]{\mathrm{BP}\langle #1 \rangle}
\newcommand{\THH}{\mathrm{THH}}
\newcommand{\VV}[1]{\langle v_1,\cdots,v_{#1}\rangle}
\newcommand{\Sss}[1]{\mathbb{S}[x_1,\cdots,x_{#1}]}
\newcommand{\EE}[1]{E^2_{*,*}(#1)}

\makeatletter
\@namedef{subjclassname@2020}{2020 Mathematics Subject Classification}
\makeatother


\numberwithin{equation}{section}

\title[Logarithmic Topological Hochschild homology of $\BP{n}$]{Logarithmic topological Hochschild homology of the truncated Brown--Peterson spectrum}

\author[J. Zha]{Jiaxi Zha}
\address{Department of Mathematics, Nankai University, No.94 Weijin Road, Tianjin 300071, P. R. China}
\email{1093913699@qq.com}

\begin{document}

\begin{abstract}
We study logarithmic topological Hochschild homology (log THH) of the truncated Brown--Peterson spectra $\BP{n}$. We first construct compatible prelog structures on $\BP{n}$ generated by the classes $v_i$. Then we compute the $V(n)$-homotopy groups of log THH of $\BP{n}$ for the prelog structures generated by $v_n$ and $(v_1,\cdots,v_n)$, under the assumption that the Smith--Toda complex $V(n)$ exists as a ring spectrum. At height $2$ and for $p\geq7$, we also compute the $V(2)$-homotopy groups of log THH of $\BP{2}/v_1$ relative to $v_2$.
\end{abstract}

\maketitle

\section{Introduction}
Algebraic $K$-theory is a basic invariant of rings and ring spectra, but it is often hard to compute directly. The trace method relates algebraic $K$-theory to topological Hochschild homology (THH) and topological cyclic homology, which are often more accessible. However, the trace method has some limitations for localization questions in algebraic $K$-theory. For example, THH and topological cyclic homology do not generally satisfy the same d\'evissage statements as algebraic $K$-theory, and localizations often lead to non-bounded-below spectra where the usual computational techniques based on the work of Nikolaus--Scholze \cite{NS18} are harder to apply.

To address these problems, Rognes \cite{Rog09} introduced prelog and logarithmic structures on ring spectra and constructed logarithmic topological Hochschild homology for such structured ring spectra, using ideas from logarithmic geometry. Rognes, Sagave and Schlichtkrull further developed log THH. In \cite{RSS15}, they established localization sequences for log THH, and in \cite{RSS18}, they computed the log THH of the Adams summand and connective complex $K$-theory. A closely related construction is the relative term $\THH(A|x)$, which goes back to the work of Hesselholt--Madsen \cite{HM03}. Recently, Lundemo \cite{Lun25} proved that, under suitable hypotheses, this construction can be recovered from log THH. This illustrates the flexibility of the logarithmic framework. Another interesting application of log THH comes from its relation to adjunctions of roots, where logarithmic methods give useful information about THH and algebraic $K$-theory; see [4, 5]. Ausoni, Bayindir and Moulinos also developed methods which extend constructions previously available only for $\Ee_\infty$-algebras to the setting of $\Ee_1$-$\Ss[\sigma_k]$-algebras; see \cite[Definition 6.6]{ABM22}.

Our computation is based on the recent framework of Rognes--Sagave--Schlichtkrull \cite{RSS25}, which further develops this point of view. Their work extends log THH from $\Ee_\infty$-algebras to $\Ee_k$-algebras for $k\geq 2$, and constructs the cyclotomic structure needed for logarithmic topological cyclic homology. It also allows us to consider prelog structures generated by several homotopy classes, rather than only by a single homotopy class. This framework is especially well suited for truncated Brown--Peterson spectra, since $\BP{n}$ has homotopy groups concentrated in even degrees. In this paper, we use this framework to compute the $V(n)$-homotopy groups of log THH of $\BP{n}$ relative to the class $v_n$. Moreover, by an induction, we also compute $V(n)$-homotopy groups of log THH of $\BP{n}$ relative to the classes $(v_1,\cdots,v_n)$. We then specialize to height $2$ and determine the remaining terms in a two-generator repletion-residue diagram (see the diagram below). The prelog structures used below depend on the compatible choices made in  Section \ref{sec3}.

However, log THH does not in general carry an algebra structure; see \cite[\S 5.20]{RSS18}. This makes computations more difficult. For the prelog structure generated by $v_n$, we show that the resulting log THH admits a suitable algebra structure; see Proposition \ref{prop:n-1case}. We then compute the log THH of $\BP{n}$ with respect to this prelog structure. This result is analogous to the log THH computation for the Adams summand by Rognes--Sagave--Schlichtkrull \cite{RSS18}. The following is our first main result. 
\begin{thm}
	Suppose that the Smith--Toda complex $V(n)$ exists as a ring spectrum. Then there is an isomorphism of $\mathbb F_p$-algebras
\[V(n)_*\THH(\BP{n},\langle v_n\rangle)\cong E(\lambda_1,\cdots,\lambda_n,d\mathrm{log}x_n)\otimes P(\kappa_n)\]
with $|d\mathrm{log}x_n|=1$, $|\kappa_n|=2p^n$, and $|\lambda_i|=2p^i-1$ for $1\leq i\leq n$.
\end{thm}

Next, we compute the $V(n)$ homotopy groups of log THH of $\BP{n}$ with respect to the prelog structure generated by $(v_1,,\cdots,v_n)$. In this case, the corresponding log THH spectra do not have an evident algebra structure, so the computation becomes more difficult.  We therefore proceed by iterated logarithmic base change, repleting the generators one at a time. This allows us to reduce the multi-generator calculation to the already understood single-generator calculation. As a result, we obtain the following result.

\begin{thm}\label{thm:1.2}
Suppose that the Smith--Toda complex $V(n)$ exists as a ring spectrum. Then there is an isomorphism of $\Ff$-modules
    \[V(n)_*\THH(\BP{n},\VV{n}) \cong E(\lambda_1,d\mathrm{log} x_1,\cdots,d\mathrm{log} x_n)\otimes P(\mu_{n+1}) \otimes P_p([dx_1],\cdots,[dx_n]),\]
    where $|d\mathrm{log} x_i|=1,|\mu_{n+1}|=2p^{n+1},|[dx_i]|=2p^i$, and $|\lambda_1|=2p-1$ for $1\leq i\leq n$.
\end{thm}

Finally, we consider the height two case. By \cite[Example 10.8]{RSS25}, there is a commutative diagram 
\[
\xymatrix{
\THH(\BP{2}) \ar[r] \ar[d] & \THH(\BP{2},\langle v_1\rangle) \ar[r] \ar[d]
& \Sigma \THH(\BP{2}/v_1) \ar[d] &\\
\THH(\BP{2},\langle v_2\rangle) \ar[r] \ar[d] & \THH(\BP{2},\langle v_1,v_2\rangle) \ar[r] \ar[d] & \Sigma \THH(\BP{2}/v_1,\langle v_2\rangle) \ar[d] &\\
\Sigma \THH(\BP{1}) \ar[r] & \Sigma \THH(\BP{1},\langle v_1\rangle) \ar[r] & \Sigma^2 \THH(\Zz_{(p)})
}
\]
with all rows and columns being cofiber sequences. By the known computations of ordinary THH and the preceding two theorems, together with the computation of $V(2)_*\THH(\BP{2},\langle v_1\rangle)$ obtained as a by-product of the proof of Theorem \ref{thm:1.2}, the $V(2)$-homotopy groups of all spectra in the diagram are now determined except for the upper two terms in the right-hand column. We compute the upper-right term in Section \ref{sec5} and then use it to determine the middle-right term, which gives our final main theorem.

\begin{thm}
	Let $p\geq7$. There is an isomorphism of $\Ff$-modules
	\[V(2)_*\THH(\BP{2}/v_1,\langle v_2\rangle)\cong E(\lambda_1,\epsilon_1,d\mathrm{log}x_2)\otimes P(\mu_3)\otimes P_p([dx_1])\otimes P_p([dx_2])\]
	with $|\lambda_1|=|\epsilon_1|=2p-1$, $|d\mathrm{log}x_2|=1$, $|\mu_3|=2p^3$, $|[dx_1]|=2p$, and $|[dx_2]|=2p^2$.
\end{thm}

\begin{rem}
	Note that $\BP{2}/v_1$ is only an $\Ee_1$-ring spectrum. Thus the corresponding log THH is not defined in the framework of \cite{RSS25}. We instead use the definition of log THH from \cite{ABM22}. With this definition, the right column is still a cofiber sequence.
\end{rem}

\subsection*{Organization}
In Section \ref{sec2}, we recall the definition of log THH and record the formulas needed for the subsequent computations. In Section \ref{sec3}, we describe three prelog structures on $\BP{n}$ and study the relations among them. In Section \ref{sec4}, we compute the $V(n)$-homotopy groups of log THH of $\BP{n}$ for the prelog structures generated by $v_n$ and $(v_1,\cdots,v_n)$. In Section \ref{sec5}, we carry out the explicit computations in the case $n=2$. Finally, in Section \ref{sec6}, we study maps in the repletion-residue cofiber sequences.

\subsection*{Conventions}
We work with $\infty$-categories throughout, following \cite{HTT,HA}. Therefore, all categories are understood to be $\infty$-categories unless otherwise specified. Moreover,
\begin{enumerate}
\item We write $\mathcal{S}$ for the category of spaces with its Cartesian symmetric monoidal structure, and $\mathrm{Sp}$ for the category of spectra with its smash product symmetric monoidal structure. The sphere spectrum is denoted by $\mathbb{S}$, and we abbreviate $\Ee_k$-ring spectra to $\Ee_k$-rings. For an ordinary commutative ring $R$, the associated Eilenberg--MacLane spectrum is again denoted by $R$.
\item For a symmetric monoidal category $\mathcal{C}$, we equip $\mathrm{Alg}_{\Ee_1}(\mathcal{C})$ with the symmetric monoidal structure
given by pointwise tensor product. Given another symmetric monoidal category $\mathcal{D}$, we regard $\mathrm{Fun}(\mathcal{C},\mathcal{D})$ as a symmetric monoidal category via the Day convolution.
\item For a prime $p$, we write $P(x_i)$, $E(x_i)$, and $\Gamma(x_i)$ for the polynomial, exterior, and divided power algebras on the classes $x_i$ over $\mathbb F_p$, respectively. In the spectral sequence computations, we write $a \doteq b$ to mean that $a=ub$ in the relevant term of the spectral sequence for some $u\in \mathbb F_p^\times$.
\end{enumerate}

\section{Logarithmic topological Hochschild homology}\label{sec2}
We begin by recalling the notions of prelog rings and logarithmic topological Hochschild homology, first introduced for $\Ee_\infty$-ring spectra in \cite{Rog09,RSS15,RSS18}. These notions were later extended in \cite{RSS25} to even ring spectra which need not be $\Ee_\infty$. Since the main examples in this paper, notably $\BP{n}$, fall into this latter setting, we use the formulation of \cite{RSS25}. By \cite[Corollary B.6]{RSS25}, for a prelog $\Ee_\infty$-ring whose prelog structure is generated by a single element, the associated logarithmic topological Hochschild homology agrees with that defined in \cite{Rog09,RSS15,RSS18}.

\begin{Def}\cite[Definition 4.1]{RSS25}
	A prelog $\Ee_k$-ring $(A,\alpha,\bar{\alpha})$ consists of an $\Ee_k$-ring $A$, an $\Ee_k$-map $\alpha\colon M\to \mathrm{Pic}(\Ss)$, and an $\Ee_k$-ring map $\bar{\alpha}:\mathrm{Th}(\alpha)\to A.$ 
\end{Def}
\begin{exmp}
	The Thom spectrum $\mathrm{Th}(\alpha)$ carries a canonical prelog $\Ee_k$-ring structure given by $(\mathrm{Th}(\alpha),\alpha,\mathrm{id}).$
\end{exmp}

Let $I$ be a discrete commutative monoid, and let $M\to I$ be an $\Ee_k$-map of spaces. Given an $\Ee_k$-map of spaces $\alpha\colon M\to \mathrm{Pic}(\Ss)$, we view it as an $\Ee_k$-map over $I$
\[\alpha\colon M\to \mathrm{Pic}(\Ss)\times I.\]
By \cite[\S3.3.2]{HA}, this corresponds to a map of $\Ee_k$-algebras
\[\alpha_*\colon M_*\to c^*\mathrm{Pic}(\Ss)\]
in $\mathrm{Fun}(I,\mathcal{S})$, where $c^*\mathrm{Pic}(\Ss)$ denotes the constant functor. Using \cite[Lemma 2.14]{RSS25}, we shall also regard $\alpha_*$ as an $\mathbb E_k$-algebra in $\mathrm{Fun}(I,\mathcal{S}_{/\mathrm{Pic}(\Ss)})$. Since $\mathrm{Pic}(\Ss)$ is grouplike, the $\Ee_k$-map $\alpha:M\to \mathrm{Pic}(\Ss)$ extends to an $\Ee_k$-map 
\[\alpha^{\mathrm{gp}}\colon M^{\mathrm{gp}}\to \mathrm{Pic}(\Ss),\]
by \cite[Corollary 3.5]{RSS25}. Applying the previous construction to $\alpha^{\mathrm{gp}}$ yields an $\Ee_k$-algebra
\[
\alpha_*^{\mathrm{gp}}
\in
\mathrm{Fun}(I^{\mathrm{gp}},
\mathcal S_{/\mathrm{Pic}(\Ss)}).
\]
This in turn gives an $I^{\mathrm{gp}}$-graded $\Ee_{k-1}$-ring $\THH(\mathrm{Th}(\alpha_*^{\mathrm{gp}})),$ obtained as the image of $\alpha_*^{\mathrm{gp}}$ under the composite functor
\[\THH: \mathrm{Alg}_{\Ee_k}(\mathrm{Fun}(I^{\mathrm{gp}},\mathcal{S}_{/\mathrm{Pic}(\Ss)}))\xrightarrow{\mathrm{Th}\circ-}\mathrm{Alg}_{\Ee_k}(\mathrm{Fun}(I^{\mathrm{gp}},\mathrm{Sp}))\xrightarrow{\THH}\mathrm{Alg}_{\Ee_{k-1}}(\mathrm{Fun}(I^{\mathrm{gp}},\mathrm{Sp})).\]
Moreover, we denote the natural map $I\to I^{\mathrm{gp}}$ by $\gamma_I$.

\begin{Def}\cite[Definition 5.16 and 5.21]{RSS25}
	Assume $k\geq 2$. The $I$-graded logarithmic topological Hochschild homology of the canonical prelog $\Ee_k$-ring $(\mathrm{Th}(\alpha),\alpha,\mathrm{id})$ is the pullback along $\gamma_I$
	\[\THH(\mathrm{Th}(\alpha_*),\alpha):=\gamma^*_I \THH(\mathrm{Th}(\alpha_*^{\mathrm{gp}}))= \THH(\mathrm{Th}(\alpha_*^{\mathrm{gp}}))\circ\gamma_I.\]
	By \cite[Equation 5.6]{RSS25}, there is a map of $I$-graded $\Ee_{k-1}$-rings
	\[\THH(\mathrm{Th}(\alpha))\to\THH(\mathrm{Th}(\alpha_*),\alpha).\]
	The logarithmic topological Hochschild homology of a prelog $\Ee_k$-ring $(A,\alpha,\bar{\alpha})$ is then defined to be the $\Ee_{k-2}$-ring
	\[\THH(A,\alpha,\bar{\alpha}):=\THH(A)\otimes_{\THH(\mathrm{Th}(\alpha))}\THH(\mathrm{Th}(\alpha_*),\alpha).\]
\end{Def}

Now, let $N\to J$ be another $\Ee_k$-map of spaces with $J$ a discrete commutative monoid, and let $\beta\colon N\to \mathrm{Pic}(\Ss)$ be an $\Ee_k$-map of spaces. The above construction yields an $\Ee_k$-algebra $\beta_*\in\mathrm{Fun}(J,\mathcal{S}_{/\mathrm{Pic}(\Ss)})$. Since $\mathcal{S}_{/\mathrm{Pic}(\Ss)}$ is symmetric monoidal, the external
product of $\alpha_*$ and $\beta_*$, followed by the monoidal product in
$\mathcal{S}_{/\mathrm{Pic}(\Ss)}$, defines an $\mathbb E_k$-algebra
\[\alpha_*\boxtimes \beta_* \in
\mathrm{Fun}(I\times J,\mathcal{S}_{/\mathrm{Pic}(\Ss)}).
\]
Explicitly, it is the composite
\[I\times J \xrightarrow{\alpha_*\times\beta_*} \mathcal{S}_{/\mathrm{Pic}(\Ss)}\times \mathcal{S}_{/\mathrm{Pic}(\Ss)} \xrightarrow{\times} \mathcal{S}_{/\mathrm{Pic}(\Ss)}.\]
Equivalently, $\alpha_*\boxtimes \beta_*$ can be obtained by applying the
above construction to the $\mathbb E_k$-map
\[\alpha\boxtimes\beta\colon M\times N \xrightarrow{\alpha\times \beta} \mathrm{Pic}(\Ss) \times \mathrm{Pic}(\Ss) \xrightarrow{\otimes} \mathrm{Pic}(\Ss). \]
Similarly, the symmetric monoidal structure on $\mathrm{Sp}$ induces an
external tensor product
\[
\mathrm{Fun}(I,\mathrm{Sp})\times \mathrm{Fun}(J,\mathrm{Sp})
\longrightarrow
\mathrm{Fun}(I\times J,\mathrm{Sp}),
\qquad
(X,Y)\longmapsto X\otimes Y,
\]
given pointwise by
\[
(X\otimes Y)_{(i,j)}\simeq X_i\otimes Y_j.
\]
This external
tensor product sends a pair consisting of an $I$-graded $\Ee_k$-algebra and a $J$-graded $\Ee_k$-algebra to an $I\times J$-graded $\Ee_k$-algebra.

\begin{lem}\label{lem:gd}
    With the above notation, there is an equivalence of $(I\times J)$-graded $\Ee_{k-1}$-algebras
	\[\THH(\mathrm{Th}(\alpha_*\boxtimes\beta_*),\alpha\boxtimes\beta)\simeq \THH(\mathrm{Th}(\alpha_*),\alpha)\otimes\THH(\mathrm{Th}(\beta_*),\beta).\]
\end{lem}
\begin{proof}
Since group completion, the Thom spectrum functor, and topological
Hochschild homology are symmetric monoidal, we have a natural equivalence of  $I^{\mathrm{gp}}\times J^{\mathrm{gp}}$-graded $\Ee_{k-1}$-algebras
\begin{equation}\label{eq:gd}
	\begin{split}
	    \THH(\mathrm{Th}((\alpha_*\boxtimes\beta_*)^{\mathrm{gp}}))
		& \simeq \THH(\mathrm{Th}(\alpha_*^{\mathrm{gp}}\boxtimes\beta_*^{\mathrm{gp}}))\\
	    & \simeq \THH(\mathrm{Th}(\alpha_*^{\mathrm{gp}})\otimes\mathrm{Th}(\beta_*^{\mathrm{gp}}))\\
		& \simeq \THH(\mathrm{Th}(\alpha_*^{\mathrm{gp}}))\otimes\THH(\mathrm{Th}(\beta_*^{\mathrm{gp}})).
	\end{split}
\end{equation}
Moreover, the following diagram commutes
\[
\xymatrix{
\mathrm{Fun}(I^{\mathrm{gp}},\mathrm{Sp})\times \mathrm{Fun}(J^{\mathrm{gp}},\mathrm{Sp}) \ar[r]^-{\gamma_I^*\times\gamma_J^*} \ar[d]_-{\otimes}
&\mathrm{Fun}(I,\mathrm{Sp})\times \mathrm{Fun}(J,\mathrm{Sp}) \ar[d]^-{\otimes} &\\
\mathrm{Fun}(I^{\mathrm{gp}}\times J^{\mathrm{gp}},\mathrm{Sp}) \ar[r]^-{\gamma_{I\times J}^*} 
&\mathrm{Fun}(I\times J,\mathrm{Sp}),
}\]
since both compositions send $(X,Y)$ to 
\[I\times J\rightarrow I^{\mathrm{gp}}\times J^{\mathrm{gp}}\xrightarrow{X\times Y} \mathrm{Sp}\times \mathrm{Sp}\xrightarrow{\otimes}\mathrm{Sp}.\]
Consequently, the equivalence \eqref{eq:gd} implies that
\begin{equation*}
	\begin{split}
		\THH(\mathrm{Th}(\alpha_*),\alpha)\otimes\THH(\mathrm{Th}(\beta_*),\beta) & = \gamma_I^*\THH(\mathrm{Th}(\alpha_*^{\mathrm{gp}}))\otimes\gamma_J^*\THH(\mathrm{Th}(\beta_*^{\mathrm{gp}}))\\
		& \simeq \gamma_{I\times J}^* [\THH(\mathrm{Th}(\alpha_*^{\mathrm{gp}}))\otimes\THH(\mathrm{Th}(\beta_*^{\mathrm{gp}}))]\\
		& \simeq \gamma_{I\times J}^* \THH(\mathrm{Th}((\alpha_*\boxtimes\beta_*)^{\mathrm{gp}}))\\
		& = \THH(\mathrm{Th}(\alpha_*\boxtimes\beta_*),\alpha\boxtimes\beta). \qedhere
	\end{split}
\end{equation*}
\end{proof}

Recall from \cite[Construction 4.4]{RSS25} that the $\Zz^n_{\geq0}$-graded $\Ee_2$-algebra $\Ss[t_1,\cdots,t_n]$ is defined as the product of $\Zz^n$-graded Thom spectrum associated to $\Ee_2$-maps 
\[\mathrm{id}_{\Zz_{\geq0}^n}\colon\Zz_{\geq0}^n\to \Zz_{\geq0}^n\]
and
\[\xi_{a_1,\cdots,a_n}\colon\Zz_{\geq0}^n\xrightarrow{(a_1,\cdots,a_n)}\Zz^n\longrightarrow\mathrm{Pic}(\Ss)^n\xrightarrow{\otimes}\mathrm{Pic}(\Ss)\]
where $a_i$ is a nonnegative even integer for $1\leq i\leq n$. Here the second map is the $n$-fold product of the $\Ee_2$-map constructed in \cite[Corollary 7.2.5]{Lur15}. The inclusion $\{1\}\hookrightarrow\Zz_{\geq0}^n$ determines a map $\bigoplus_{i}\Ss^{a_i}\to \Ss[t_1,\cdots,t_n]$. By adjunction, this induces a map of $\Ee_2$-rings
\[\mathrm{Free}_{\Ee_2}(\bigoplus_{i}\Ss^{a_i})\to \Ss[t_1,\cdots,t_n].\] 
By \cite[Proposition 4.6]{RSS25} (see also \cite[Proposition 3.11]{ABM22}), this map is the colimit of a sequence
\begin{equation}\label{eq:attach}
	\mathrm{Free}_{\Ee_2}(\bigoplus_{i}\Ss^{a_i})=X_0\to X_1\to\cdots\to \mathrm{colim}_k X_k\simeq \Ss[t_1,\cdots,t_n]
\end{equation}
in the category of $\Zz_{\geq0}^n$-graded $\Ee_2$-rings where each stage is obtained from the previous via attaching finitely many even cells in the category of $\Zz_{\geq0}^n$-graded $\Ee_2$-rings.

We simply write $\THH(\Ss[t_1^{\pm1},\cdots,t_n^{\pm1}])_{\geq0}$ for the $\Zz^n_{\geq 0}$-graded $\Ee_1$-ring $\THH(\mathrm{Th}((\xi_{a_1,\cdots,a_n})_*),\xi_{a_1,\cdots,a_n})$, and denote its $(0,\cdots,0)$-graded summand by $\THH(\Ss[t_1^{\pm1},\cdots,t_n^{\pm1}])_{=0}$. We end this section by the following proposition which computes their mod $p$ homology.
\begin{prop}\label{prop:S[x]}
	There are isomorphisms of algebras
    \begin{equation}\label{eq:pic}
        H_*(\THH(\Ss[t_1,\cdots,t_n]);\Ff)\cong E(dt_1,\cdots,dt_n)\otimes P(t_1,\cdots,t_n),
    \end{equation}
	\begin{equation}\label{eq:geq0}
	H_*(\THH(\Ss[t_1^{\pm1},\cdots,t_n^{\pm1}])_{\geq0};\Ff)\cong E(d\mathrm{log}t_1,\cdots,d\mathrm{log}t_n)\otimes P(t_1,\cdots,t_n),
	\end{equation}
	and
	\begin{equation}\label{eq:=0}
		H_*(\THH(\Ss[t_1^{\pm1},\cdots,t_n^{\pm1}])_{=0};\Ff)\cong E(d\mathrm{log}t_1,\cdots,d\mathrm{log}t_n)
	\end{equation}
	with $|t_i|=a_i$, $|dt_i|=a_i+1$, and $|d\mathrm{log}t_i|=1$ for $1\leq i \leq n$.
\end{prop}
\begin{proof}
	By \cite[Proposition 11.4]{RSS25} and the K\"unneth
theorem, there is an isomorphism of $\Zz_{\geq0}$-graded algebras
\[H_*(\THH(\Ss[t_i]);\Ff)\cong E(dt_i)\otimes P(t_i)\]
and an isomorphism of $\Zz$-graded algebras
\[H_*(\THH(\Ss[t_i^{\pm1}]);\Ff)\cong E(d\mathrm{log}t_i)\otimes P(t_i^{\pm1})\]
with $|t_i|=a_i$, $|dt_i|=a_i+1$, and $|d\mathrm{log}t_i|=1$ for $1\leq i\leq n.$ The associated Tor spectral sequences yields the first isomorphism and an
isomorphism of $\Zz^n$-graded algebras
	\[H_*(\THH(\Ss[t_1^{\pm1},\cdots,t_n^{\pm1}]);\Ff)\cong E(d\mathrm{log}t_1,\cdots,d\mathrm{log}t_n)\otimes P(t_1^{\pm1},\cdots,t_n^{\pm1}).\]
	Restricting to the $\Zz_{\geq0}^n$-graded summand yields the second isomorphism by Lemma \ref{lem:gd}. Restricting further to the
$(0,\cdots,0)$-graded summand yields the third isomorphism.
\end{proof}

\section{Prelog structures on \texorpdfstring{$\BP{n}$}{BP<n>}}\label{sec3}

In this section, we assume $n\geq2$ and recall the prelog structures on $\BP{n}$. To this end, by \cite[Theorem 2.0.6]{HW22}, we fix an $\Ee_3$ form of
$\BP{n}$. Moreover, by the construction in the proof of
\cite[Theorem 2.0.6]{HW22}, we may also choose $\Ee_3$ forms of
$\BP{i}$, for $1\leq i\leq n$, such that each map
\[ f_i\colon \BP{i}\longrightarrow \BP{i-1} \]
is an $\Ee_3$-ring map, and the induced map on homotopy groups is the canonical quotient
\[(f_i)_*\colon \pi_*\BP{i} \cong \mathbb Z_{(p)}[v_1,\cdots,v_i] \longrightarrow \pi_*\BP{i-1} \cong \mathbb Z_{(p)}[v_1,\cdots,v_{i-1}],\]
sending $v_i$ to $0$. 

Next, for $1\leq i\leq n$, let $\Ss[x_i]$ denote the $\Zz_{\geq0}$-graded Thom spectrum associated to $\Ee_2$-maps $\mathrm{id}_{\Zz_{\geq0}}\colon\Zz_{\geq0}\to \Zz_{\geq0}$ and 
\begin{equation}\label{eq:xi}
	\xi_{2p^i-2}\colon\Zz_{\geq0}\xrightarrow{2p^i-2}\Zz\to\mathrm{Pic}(\Ss).
\end{equation}
Hence a choice of representative $v_i:\mathbb S^{2p^i-2}\longrightarrow \BP{i}$ of the class $v_i\in \pi_{2p^i-2}\BP{i}$ extends to an $\Ee_2$-ring map 
\[\bar{v}_i\colon\Ss[x_i]\to \BP{i},\]
by Equation \eqref{eq:attach}. Note that the degree-zero summand of $\Ss[x_i]$ is $\Ss$.

The map $\bar{v}_i$ is not canonical, since it depends on the choices made in the construction. This ambiguity will not play a role in the general discussion below. However, in order to carry out the explicit computations of Section \ref{sec4}, we will fix a particular choice of $\bar{v}_i$.

Let $J_{\mathbb{C}}\colon\mathrm{BU}\times \Zz\to \mathrm{Pic}(\Ss)$ be the $J$-homomorphism. Its adjoint map $j\colon\Zz\to \mathrm{Fun}(\mathrm{BU},\mathrm{Pic}(\Ss))$ is lax symmetric monoidal, by \cite[\S3]{LNP25}. Therefore, the following composition
\[F\colon\Zz_{\geq0}\xrightarrow{2p^i-2}\Zz\xrightarrow{j} \mathrm{Fun}(\mathrm{BU},\mathrm{Pic}(\Ss))\xrightarrow{\mathrm{Th}}\mathrm{Sp}\]
is lax symmetric monoidal, since the Thom spectrum functor is symmetric monoidal. On the other hand, we can also consider the functor
\[G\colon\Zz_{\geq0}\xrightarrow{\xi_{2p^i-2}}\mathrm{Pic}(\Ss)\xrightarrow{-\otimes \mathrm{MU}}\mathrm{Pic}(\mathrm{MU})\to\mathrm{Sp},\]
where the first functor is Functor \eqref{eq:xi} and the last functor is the forgetful functor which is lax symmetric monoidal. By \cite[p. 85]{HL18}, $(-\otimes\mathrm{MU})\circ\xi_{2p^n-2}$ is an $\Ee_{\infty}$-map of spaces. Hence $G$ is also lax symmetric monoidal.

\begin{lem}\label{lem:MU and S}
	The lax symmetric monoidal functors $F$ and $G$ are equivalent.
\end{lem}
\begin{proof}
	To see that $F$ and $G$ are equivalent as functors, it suffices to observe that, for any $k\geq 0$,
	\[F(k)\simeq G(k)\simeq \Sigma^{(2p^i-2)k}\mathrm{MU},\]
	since $\Zz_{\geq0}$ is discrete as a category. Moreover, the lax symmetric monoidal structures on both functors are induced by the product on $\mathrm{MU}$. Hence $F$ and $G$ are equivalent as lax symmetric monoidal functors.
\end{proof}
By \cite[Proposition 2.6.2]{HW22}, there is an $\Ee_3$-ring map $g_i\colon\mathrm{MU}[x_i]\to \BP{i}$ such that $\pi_*(g_i)(x_i)=v_i$, where $\mathrm{MU}[x_i]$ is the colimit of $F$. Since the functor $-\otimes\mathrm{MU}$ commutes with colimit, there is an equivalence of $\Ee_2$-rings $\mathrm{MU}[x_i]\simeq \Ss[x_i]\otimes\mathrm{MU}$ by Lemma \ref{lem:MU and S}. Consequently, the underlying $\Ee_2$-ring map of $g$ gives rise to an $\Ee_2$-ring map $\bar{v}_i\colon \Ss[x_i]\to \BP{i}$. This is the choice of $\bar{v}_i$ mentioned above.

\begin{lem}\label{lem:aug}
	For each $1\leq i\leq n,$ there exists an $\Ee_2$-augmentation
	\[\epsilon\colon\Ss[x_i]\to \Ss\]
	such that the composition of $\Ee_2$-ring map
	\[\Ss[x_i]\to \BP{i}\to \BP{i-1}\]
	factors as
	\[\Ss[x_i]\xrightarrow{\epsilon} \Ss \xrightarrow{\eta} \BP{i-1},\]
	where $\eta$ is the unit map of $\BP{i-1}$.
\end{lem}
\begin{proof}
	By \cite[Lemma B.0.6]{HW22}, there is a symmetric monoidal localization
	\[\mathrm{Fun}(\Zz_{\geq0},\mathrm{Sp})\to \mathrm{Sp},\]
	which sends a $\Zz_{\geq0}$-graded spectrum its degree-zero summand. Since $\Ss[x_i]$ is an $\Ee_2$-algebra in $\mathrm{Fun}(\Zz_{\geq0},\mathrm{Sp})$ whose degree-zero summand is $\Ss$, the unit of this localization gives the desired $\Ee_2$-augmentation.
	
	We prove the asserted factorization by induction. For the initial step, since $f_*(v_i)=0$, the composite
	\[\mathrm{Free}_{\Ee_2}(\Ss^{2p^i-2})\to \BP{i}\to \BP{i-1}\]
	is null. Hence it factors through
	\[\mathrm{Free}_{\Ee_2}(\Ss^{2p^i-2})\to \mathrm{Free}_{\Ee_2}(0)\simeq \Ss\to\BP{i-1}.\]
	Now, suppose that $X_n\to \BP{i}\to \BP{i-1}$ factors through the unit map $\Ss\to\BP{i-1}$. Recall that $X_{n+1}$ is obtained as the pushout of the span 
	\[\Ss\leftarrow \mathrm{Free}_{\Ee_2}(\bigoplus_{k\in I_n}\Ss^k)\xrightarrow{f} X_n\]
	in $\Ee_2$-rings, where $I_n$ is a finite set consisting of odd positive numbers. Since $\BP{i-1}$ is even, any $\Ee_2$-ring map $\mathrm{Free}_{\Ee_2}(\Ss^k)\xrightarrow{f}\ \BP{i-1}$ is null, and hence factors through $\mathrm{Free}_{\Ee_2}(0)\simeq \Ss$. Together with the induction hypothesis, this gives the following commutative diagram in $\Ee_2$-rings:
\[
\begin{tikzcd}[row sep=0.8em, column sep =0.5em]
   & \Ss \arrow[dd, dashed] \arrow[dl] && \mathrm{Free}_{\Ee_2}(\bigoplus_{k\in I_n}\Ss^k) \arrow[ll] \arrow[rr] \arrow[dl] \arrow[dd, dashed] &&
   X_n \arrow[dd] \arrow[dl] \\
    \Ss \arrow[dd] && \Ss \arrow[ll] \arrow[rr] \arrow[dd] &&
   \Ss \arrow[dd] \\
   & \BP{i} \arrow[dl, dashed] && \BP{i} \arrow[ll,dashed]        
   \arrow[rr, dashed ] \arrow[dl, dashed] && \BP{i} \arrow[dl] \\
   \BP{i-1} && \BP{i-1} \arrow[ll] \arrow[rr] && \BP{i-1}.
\end{tikzcd}
\]
Taking pushouts of the horizontal spans, we obtain the desired result.
\end{proof}

We also require the $\bar{v}_i$ to be mutually compatible for $1\leq i\leq n$. To this end, let $\Ss[t]$ be the Thom spectrum associated to the $\Ee_2$-map of spaces $\Zz_{\geq0}\xrightarrow{a}\Zz\to \mathrm{Pic}(\Ss),$ where $a$ is a positive even integer.

\begin{lem}\label{lem:lift}
    Give $1\leq i\leq n$, let $f\colon\Ss[t]\to\BP{i-1}$ be an $\Ee_2$-ring map. Then $f$ admits an $\Ee_2$-lift
    \[\widetilde{f}\colon\Ss[t]\to\BP{i}\]
    such that $f_i\circ\widetilde{f}\simeq f.$
\end{lem}
\begin{proof}
    Let $\Ss[t]$ be the colimit of a sequence $\{X_n\}_{n\geq0}$ in Equation \eqref{eq:attach}. By our choice of $f_i$, the induced map
    \[\pi_a(\BP{i})\cong \pi_0(\mathrm{Map}_{\Ee_2}(\mathrm{Free}_{\Ee_2}(\Ss^a),\BP{i})) \to \pi_0(\mathrm{Map}_{\Ee_2}(\mathrm{Free}_{\Ee_2}(\Ss^a),\BP{i-1}))\cong\pi_a(\BP{i-1})\]
    is surjective. It follows that the map $X_0=\mathrm{Free}_{\Ee_2}(\Ss^a)\to \BP{i-1}$ admits a lift along $f_i$ to $\BP{i}$. Next, suppose that $X_{n+1}$ is the pushout of the span
    \[\Ss\leftarrow\mathrm{Free}_{\Ee_2}(\bigoplus_{n_i\geq 0}S^{2n_i+1})\to X_n\]
    in the category of $\Ee_2$-rings. Denote the induced map $X_n\to X_{n+1}$ by $i_n$. Then we obtain a pullback of spaces
\[
\xymatrix{
\mathrm{Map}_{\Ee_2}(X_{n+1},\BP{i-1}) \ar[r]^-{i_n^*} \ar[d]
& \mathrm{Map}_{\Ee_2}(X_n,\BP{i-1}) \ar[d] &\\
\mathrm{Map}_{\Ee_2}(\Ss,\BP{i-1}) \ar[r] & \mathrm{Map}_{\Ee_2}(\mathrm{Free}_{\Ee_2}(\bigoplus_{n_i\geq 0}\Ss^{2n_i+1}),\BP{i-1}).}
\]
Since $\Ss$ is the unit $\Ee_2$-ring, $\mathrm{Map}_{\Ee_2}(\Ss,\BP{n-1})$ is contractible. Moreover, there is an equivalence
\begin{equation*}
    \begin{split}
        \mathrm{Map}_{\Ee_2}(\mathrm{Free}_{\Ee_2}(\bigoplus_{n_i\geq 0}\Ss^{2n_i+1}),\BP{i-1}) & \simeq \prod_{n_i\geq 0}\mathrm{Map}_{\Ee_2}(\mathrm{Free}_{\Ee_2}(\Ss^{2n_i+1)},\BP{i-1})\\
        & \simeq \prod_{n_i\geq 0}\mathrm{Map}(\Ss^{2n_i+1},\BP{i-1}).
    \end{split}
\end{equation*}
Thus, the fiber of $i_n^*$ over any $\Ee_2$-ring map is $\prod_{n_i\geq 0}\mathrm{Map}(\Ss^{2n_i+2},\BP{i-1}).$ A similar discussion applies to $\BP{i}$. Now, suppose that the $\Ee_2$-ring map $X_n\to \BP{i-1}$ has been lifted along $f_i$ to $\BP{i}$. Then we obtain the following commutative diagram, in which both the upper and lower rows are fiber sequences.
\[
\xymatrix{
\prod_{n_i\geq 0}\mathrm{Map}(\Ss^{2n_i+2},\BP{i}) \ar[r] \ar[d]_-{(f_i)_*} &  \mathrm{Map}_{\Ee_2}(X_{n+1},\BP{i}) \ar[r]^-{i_n^*} \ar[d]_-{(f_i)^*}
& \mathrm{Map}_{\Ee_2}(X_n,\BP{i}) \ar[d]^-{(f_i)_*} &\\
\prod_{n_i\geq 0}\mathrm{Map}(\Ss^{2n_i+2},\BP{i-1}) \ar[r] & \mathrm{Map}_{\Ee_2}(X_{n+1},\BP{i-1}) \ar[r]^-{i_n^*} & \mathrm{Map}_{\Ee_2}(X_n,\BP{i-1})}
\]
With $f_i$ chosen as above, the leftmost column induces a surjection on $\pi_0$. A diagram chase then shows that the $\Ee_2$-ring map $X_{n-1}\to \BP{i-1}$ also admits a lift along $f_i$ to $\BP{i}$. Passing to the limit, we obtain the desired lift.
\end{proof}

\begin{rem}
	The same argument implies that there exists an $\Ee_2$-lift of $g_i$ along $f_{i+1}$ for $i\geq 1$. However, note that there is no $\Ee_3$-ring map $g\colon\mathrm{MU}[x_i]\to \BP{i+1}$ such that $\pi_*(g)(x_i)=v_i$. Otherwise, it would induce an $\Ee_2$-ring map $\mathcal{U}^{(1)}_{\mathrm{MU}}(\Ff)\to\mathcal{U}^{(1)}_{\BP{i+1}}(\Ff).$ Hence applying the Dyer--Lashof operation yields that
	\[0=g_*(Q^{p^i}(\sigma x_i))=Q^{p^i}(g_*(\sigma x_i))=Q^{p^i}(\sigma v_i)=\sigma v_{i+1},\]
	by \cite[Lemma 2.4.1]{HW22}, a contradiction.
\end{rem}

Therefore, along the composite $f_{i+1}\circ\cdots\circ f_n$, the map $\bar{v}_i:\Ss[x_i]\to\BP{i}$ can be lifted to a map $\widetilde{v}_i:\Ss[x_i]\to\BP{n}$ such that $\pi_*(\widetilde{v}_i)(x_i)=v_i$ for $1\leq i\leq n-1$. By the Dunn additivity (see \cite[Theorem 5.1.2.2]{HA}), we may view $\BP{n}$ as an $\Ee_1$-algebra in the category of $\Ee_2$-rings. In particular, the multiplication $\mu\colon\BP{n}\otimes\cdots\otimes\BP{n}\to \BP{n}$ is a map of $\Ee_2$-rings. Consequently, we obtain $\Ee_2$-ring maps
\[\bar{\alpha}_J\colon\bigotimes_{j\in J}\Ss[x_j]\xrightarrow{\bigotimes_{j\in J}\widetilde{v}_j} \bigotimes_{j\in J}\BP{n}\xrightarrow{\mu} \BP{n},\]
and 
\[\bar{\beta}_n\colon\Sss{n}\simeq\Sss{n-1}\otimes\Ss[x_n] \xrightarrow{\bar{\alpha}_{\{1,\cdots,n-1\}}\otimes \bar{v}_n} \BP{n}\otimes\BP{n}\xrightarrow{\mu}\BP{n},\]
since the Thom spectrum functor is symmetric monoidal.
\begin{Def}\label{def:logBP}
	The above construction gives rise to three prelog $\Ee_2$-ring structures on $\BP{n}$. We denote them by
	\[(\BP{n},\langle v_n\rangle):=(\BP{n},\xi_{2p^n-2},\bar{v}_n),\]
	\[(\BP{n},\VV{n}):=(\BP{n},\xi_{2p-2,\cdots,2p^n-2},\bar{\beta}_n),\]
	and
	\[(\BP{n},\langle v_j,j\in J\rangle):=(\BP{n},\xi_{\{2p^j-2,\ j\in J\}},\bar{\alpha}_J).\]
\end{Def}

\begin{rem}\label{rem:pi(a)}
	By construction, $\pi_*(\bar{\beta}_n)$ sends the generator corresponding to $x_i$ to $v_i$ for $1\leq i\leq n$. Moreover, by the Dunn additivity, $\BP{n}\otimes_{\Ss[x_n]}\Ss$ is an $\Ee_1$-ring, and there is an equivalence of spectra
	\begin{equation*}
	 \BP{n}\otimes_{\Ss[x_n]}\Ss \simeq \BP{n-1}.
	\end{equation*}
	In other words, the underlying spectrum of the $\Ee_1$-ring $\BP{n}\otimes_{\Ss[x_n]}\Ss$ is $\BP{n-1}$. Indeed, by our choice of $\bar v_n$, the above equivalence is compatible with the $\Ee_2$-ring structure obtained from the $\Ee_3$-ring map $g_n\colon \mathrm{MU}[x_n]\longrightarrow \BP{n}$ by base change along $\mathrm{MU}[x_n]\to\mathrm{MU}$. Consequently, this equivalence is an equivalence of $\Ee_2$-rings.
\end{rem}

\begin{rem}
	The composition of $\Ee_2$-ring maps $\Sss{n-1}\to\BP{n}\xrightarrow{f_n}\BP{n-1}$ defines a prelog $\Ee_2$-ring structure on $\BP{n-1}$. By the construction and Lemma \ref{lem:lift}, this prelog structure is exactly the one in Definition \ref{def:logBP}.
\end{rem}

\begin{rem}
	Lemma \ref{lem:aug} implies that there is a commutative diageram of $\Ee_2$-rings 
	\[
	\xymatrix{
	\Sss{n} \ar[r]^-{\mathrm{id}\otimes\epsilon} \ar[d]_
	{\bar{\beta}_n}
	& \Sss{n-1} \ar[d]^-{\bar{\beta}_{n-1}} &\\
	\BP{n} \ar[r]^-{f_n} &\BP{n-1}.}
	\]
	In other words, there is a map of prelog $\Ee_2$-rings $(\BP{n},\VV{n})\to (\BP{n-1},\VV{n-1})$.
\end{rem}

\section{Log THH of \texorpdfstring{$\BP{n}$}{BP<n>}}\label{sec4}
Let $p$ be a prime. From now on, we assume the existence of a $p$-local Smith--Toda complex $V(n)=\Ss/(p,v_1,\cdots,v_n)$ which admits the structure of a ring spectrum in the homotopy category of spectra. Thus $V(n)\otimes \BP{n}$ is equivalent to $\Ff$ as a $\BP{n}$-module in $\mathrm{Mod}_{\BP{n}}(\mathrm{Sp})$, and the multiplication on $V(n)$ gives $V(n)\otimes \BP{n}$ a ring structure in the homotopy category of spectra. It follows from \cite[Proposition 2.7]{AG+24} that there is an isomorphism of $\Ff$-algebras
\begin{equation}\label{eq:BP}
V(n)_*\THH(\BP{n})
\cong
\THH_*(\BP{n};\Ff)
\cong
E(\lambda_1,\cdots,\lambda_{n+1})\otimes P(\mu_{n+1}),
\end{equation}
with $|\mu_{n+1}|=2p^{n+1}$ and $|\lambda_i|=2p^i-1$ for $1\leq i\leq n+1$. The required ring structures are known for $V(0)$, $V(1)$, $V(2)$, and $V(3)$ in the ranges $p\geq 3$, $p\geq 5$, $p\geq 7$, and $p\geq 11$, respectively; see \cite{Smi70,Toda71,YY77}. In higher heights no general existence theorem is known. For example, Nave \cite[Theorem 1.3]{Nav10} proves that, for $p\geq 7$, the Smith--Toda complex $V((p+1)/2)$ does not exist, and that if $V((p-1)/2)$ exists, then it is not a ring spectrum. Consequently, all computations of $V(n)_*\THH(\BP{n},\VV{n})$ in this section are conditional on the existence of $V(n)$ with a ring spectrum structure in the homotopy category of spectra.

In this section, by abuse of notation, we use the same symbols for the corresponding classes in the spectral sequences.

\begin{rem}
It would be interesting to consider the even filtration on $\THH(\BP{n},\VV{n})$ in the sense of \cite{HRW25,Pst23b}. In this case, the quotient $\mathrm{gr}^*_{\mathrm{ev}}\Ss/(p,v_1,\cdots,v_n)$ exists for all $n$, by \cite[Theorem 6.12]{GIKR22} and \cite[Theorem 1.4 and 4.54]{Pst23a}. This suggests an associated-graded analogue of the present calculation, with coefficients in this quotient object. We hope to return to this question in future work.
\end{rem}

\subsection{The single-generator prelog structure \texorpdfstring{$\langle v_n\rangle$}{<v_n>}}
In this subsection, we adapt the computation of logarithmic topological Hochschild homology for the Adams summand in \cite{RSS18} to $\BP{n}$, with respect to the prelog structure generated by $v_n$. The computation is organized by the following vertical cospan
\begin{equation}\label{eq:cospan}
\begin{tikzcd}[row sep=1.3em]
{\THH(\BP{n})\otimes_{\THH(\Ss[x_n])}\THH(\Ss[v_n^{\pm1}])_{\geq0}}
\arrow[d]
\\
{\THH(\BP{n})\otimes_{\THH(\Ss[x_n])}\THH(\Ss[v_n^{\pm1}])_{=0}}
\\
{\THH(\BP{n})\otimes_{\THH(\Ss[x_n])}\Ss.}
\arrow[u]
\end{tikzcd}
\end{equation}
We first compute the Tor spectral sequence for the bottom term of the cospan. Naturality with respect to the two maps in the cospan then allows us to determine the differentials in the Tor spectral sequence for the top term. This gives the basic one-generator calculation which will be used repeatedly in the iterated repletion argument below.

Since the functor $\THH$ is symmetric monoidal, Remark \ref{rem:pi(a)} implies that the bottom term is equivalent to $\THH(\BP{n-1})$. Moreover, by Equation \eqref{eq:BP} and the $v_n$-Bockstein cofiber sequence defining $V(n)$, we obtain an isomorphism
\begin{equation}\label{eq:bpn-1}
	V(n)_*\THH(\BP{n-1})\cong E(\lambda_1,\cdots,\lambda_n,\epsilon_n)\otimes P(\mu_n)
\end{equation}
with $|\epsilon_n|=2p^n-1,$ $|\mu_n|=2p^n$ and $|\lambda_i|=2p^i-1$ for $1\leq i\leq n.$ Here $\epsilon_n$ denotes the class arising from the $v_n$-Bockstein. Recall that we have fixed an $\Ee_3$-form of $\BP{n}$. Consequently, the canonical map $\BP{n}\to \THH(\BP{n})$ is an $\Ee_2$-ring map which equips $\THH(\BP{n})$ with a left $\BP{n}$-module structure. Hence there is an equivalence of spectra
 \[V(n)\otimes \THH(\BP{n})\simeq (V(n)\otimes \BP{n})\otimes_{\BP{n}}\THH(\BP{n})\simeq \Ff\otimes_{\BP{n}} \THH(\BP{n}).\]
 Moreover, the natural map $f_n\colon\BP{n}\to \BP{n-1}$ is an $\Ee_3$-ring map. Iterating this process yields that $\BP{n}\to \Ff$ is also an $\Ee_3$-ring map. Note that the natural map $\Ff\otimes \THH(\BP{n})\to \Ff\otimes_{\BP{n}} \THH(\BP{n})$ is an $\Ee_1$-ring map. Pre-compositing with the $\Ee_1$-ring map 
 \[\Ff\otimes\THH(\Ss[x_n])\to \Ff\otimes \THH(\BP{n})\]
 gives $\Ff\otimes_{\BP{n}} \THH(\BP{n})$ a left $(\Ff\otimes\THH(\Ss[x_n]))$-module structure.
\begin{lem}\label{lem:ass}
    Let $A$ and $B$ be $\Ee_1$-ring spectra, and let $M$ be a left $A\otimes B$-module. Suppose that we also have $\Ee_1$-ring maps
    \[A\to k,\qquad B\to \Ss.\]
    Then there is a natural equivalence
    \[k\otimes_{k\otimes B}(k\otimes_A M)\simeq k\otimes_A(\Ss\otimes_B M).\]
\end{lem}
\begin{proof}
    Since $M$ is a left $A\otimes B$-modules, $k\otimes_A M$ is a left $k\otimes B$-modules and $\Ss\otimes_B M$ is a left $A$-module. Hence both sides of the equivalence are well-defined. By the associativity of tensor product, there is an equivalence
    \begin{equation*}
        \begin{split}
            k\otimes_{k\otimes B}(k\otimes_A M) & \simeq k\otimes_{k\otimes B}[(k\otimes B)\otimes_{A\otimes B}M]\\
            & \simeq k\otimes_{A\otimes B}M \\
            & \simeq (k\otimes_A A)\otimes_{A\otimes B}M\\
            & \simeq k\otimes_A(A\otimes_{A\otimes B} M)\\
            & \simeq k\otimes_A(\Ss\otimes_B M).\qedhere
        \end{split}
    \end{equation*}
\end{proof}

\begin{prop}\label{prop:eqV(n)}
    There is an equivalence of spectra
    \begin{equation}\label{eq:bot}
          \Ff\otimes_{\Ff\otimes\THH(\Ss[x_n])}(\Ff\otimes_{\BP{n}} \THH(\BP{n}))\simeq V(n)\otimes \THH(\BP{n-1}).  
    \end{equation}
\end{prop}
\begin{proof}
    Since the map $\Ss[x_n]\to\BP{n}$ is a $\Ee_2$-ring map and $\BP{n}$ is an $\Ee_3$-ring, the map 
    \[\Ss[x_n]\otimes\BP{n}\xrightarrow{\bar{\beta}_n\otimes\mathrm{id}}\BP{n}\otimes\BP{n}\to\BP{n}\]
    is an $\Ee_2$-ring map, where the second map is the multiplication of $\BP{n}$. By restriction of scalars along this map, we see that $\BP{n}$ admits a left $\Ss[x_n]\otimes \BP{n}$-module structure in the category of $\Ee_1$-rings. Applying $\THH$ to this structure, we obtain a left $\THH(\Ss[x_n])\otimes\THH(\BP{n})$-module structure on $\THH(\BP{n})$. Restricting the left $\THH(\BP{n})$-action along the $\Ee_1$-ring map $\BP{n}\to\THH(\BP{n})$ gives a left $\THH(\Ss[x_n])\otimes\BP{n}$-module structure on $\THH(\BP{n})$. 
    Let $(A,B,M,k)$ be $(\BP{n},\THH(\Ss[x_n]),\THH(\BP{n}),\Ff)$ in Lemma \ref{lem:ass}. We obtain that
    \[\Ff\otimes_{\Ff\otimes\THH(\Ss[x_n])}(\Ff\otimes_{\BP{n}} \THH(\BP{n}))\simeq \Ff\otimes_{\BP{n}}(\Ss\otimes_{\THH(\Ss[x_n])}\THH(\BP{n})).\]
    By Remark \ref{rem:pi(a)}, the right term is equivalent to 
    \begin{equation*}
        \begin{split}
            \Ff\otimes_{\BP{n}}\THH(\BP{n-1}) &\simeq (V(n)\otimes \BP{n})\otimes_{\BP{n}}\THH(\BP{n-1})\\
            & \simeq V(n)\otimes\THH(\BP{n-1}).\qedhere
        \end{split}
    \end{equation*}
\end{proof}
Now, the equivalence \eqref{eq:bot} gives a spectral sequence
\[\EE{\BP{n-1}}=\mathrm{Tor}_{*,*}^{H_*(\THH_*(\Ss[x_n]);\Ff)}(\Ff,V(n)_* \THH(\BP{n}))\Longrightarrow V(n)_*\THH(\BP{n-1}).\]

\begin{prop}\label{prop:n-1case}
	The above spectral sequence is multiplicative. In particular, the differentials satisfy the Leibniz rule.
\end{prop}
\begin{proof}
    By our choice of $\bar{v}_n$ and Lemma \ref{lem:ass}, there is an equivalence of spectra
    \begin{equation*}
        \begin{split}            &\Ff\otimes_{\Ff\otimes\THH(\Ss[x_n])}\Ff\otimes_{\BP{n}}\THH(\BP{n})\\
        & \simeq(\Ff\otimes \THH(\mathrm{MU}))\otimes_{\Ff\otimes\THH(\mathrm{MU}[x_n])}\Ff\otimes_{\BP{n}}\THH(\BP{n})\\
        & \simeq \Ff\otimes_{\BP{n}}(\THH(\mathrm{MU})\otimes_{\THH(\mathrm{MU}[x_n])}\THH(\BP{n})).
        \end{split}
    \end{equation*}
    An argument similar to the proof of Lemma \ref{lem:aug} gives an $\Ee_{\infty}$-augmentation $\mathrm{MU}[x_n]\to \mathrm{MU}.$ Since 
\[\THH(\BP{n})\otimes_{\THH(\mathrm{MU}[x_n])}\THH(\mathrm{MU})\simeq\THH(\BP{n}\otimes_{\mathrm{MU}[x_n]}\mathrm{MU})\simeq \THH(\BP{n-1})\]
as left-$\BP{n}$-modules, it suffices to prove that both maps in the span 
\[\Ff\otimes \THH(\mathrm{MU})\leftarrow\Ff\otimes\THH(\mathrm{MU}[x_n])\to \Ff\otimes_{\BP{n}}\THH(\BP{n})\]
are $\Ee_2$-ring maps. By definition, the map on the left is an $\Ee_\infty$-ring map. Since $\Ff\otimes_{\BP{n}} \THH(\BP{n})$ is obtained from the $\Ee_2$-$\BP{n}$-algebra $\THH(\BP{n})$ by extension of scalars along the $\Ee_3$-ring map $\BP{n} \to \Ff$, it inherits a natural $\Ee_2$-ring structure. Moreover, the unit map
\[
\THH(\BP{n})
\to
\Ff\otimes_{\BP{n}} \THH(\BP{n})
\]
is a map of $\Ee_2$-$\BP{n}$-algebras, and hence, in particular, an $\Ee_2$-ring map. Consequently, its adjoint 
\[\Ff\otimes \THH(\BP{n})\to\Ff\otimes_{\BP{n}} \THH(\BP{n})\]
is also an $\Ee_2$-ring map. Precomposing this map with $\Ff\otimes \THH(\mathrm{MU}[x_n])\to\Ff\otimes \THH(\BP{n})$, we obtain the desired $\Ee_2$-map on the left.
\end{proof}

\begin{cor}\label{cor:bpn-1}
	The $E_2$-page of $E^r_{*,*}({\BP{n-1}})$ is isomorphic to
	\[\EE{\BP{n-1}} \cong E(\lambda_1,\cdots,\lambda_{n+1})\otimes P(\mu_{n+1})\otimes E([x_n])\otimes\Gamma([dx_n]),\]
	where $\lambda_i$ has bidegree $(0,2p^i-1)$ for $1\leq i\leq n+1$, $\mu_{n+1}$ has bidegree $(0,2p^{n+1})$, $[x_n]$ has bidegree $(1,2p^n-2)$, and $[dx_j]$ has bidegree $(1,2p^n-1)$. There are non-trivial $d^p$-differentials
	\[d^p(\gamma_k([dx_n])\doteq \lambda_{n+1}\cdot\gamma_{k-p}([dx_n])\]
	for $k\geq p$, leaving
	\[E^{\infty}_{*,*}(\BP{n-1})\cong E(\lambda_1,\cdots,\lambda_n)\otimes P(\mu_{n+1})\otimes E([x_n]) \otimes P_p([dx_n]).\]
	Hence $[x_n]$ represents $\epsilon_n$ modulo $\lambda_n$, $[dx_n]$ represents $\mu_n$, and $\mu_{n+1}$ represents $\mu_n^p$. Moreover, there is a multiplicative extension $[dx_n]^p=\mu_{n+1}$.
\end{cor}
\begin{proof}
	The proof proceeds similarly to the proof of \cite[Proposition 7.2]{RSS18}. First, we identify the $E^2$-page. By Proposition \ref{prop:S[x]} and Equation \eqref{eq:bpn-1}, the algebra map
\begin{equation*}
\begin{tikzcd}[row sep=1.3em]
H_*(\THH(\Ss[x_n]);\Ff)\cong E(dx_n)\otimes P(x_n)
\arrow[d]\\
\pi_*(\Ff\otimes_{\BP{n}}\THH(\BP{n}))\cong E(\lambda_1,\cdots,\lambda_{n+1})\otimes P(\mu_{n+1})
\end{tikzcd}
\end{equation*}
sends $x_n$ to zero for degree reasons. Since the map $\THH(\Ss[x_n])\to \THH(\BP{n})$ is obtained by applying $\THH$ to the $\Ee_2$-ring map $\bar{v}_n$, it also sends $dx_n$ to zero by compatibility of the suspension operation with the induced map on $\THH$. Consequently, the $E_2$-page can be identified with
\begin{equation*}
\begin{split}
    \EE{\BP{n-1}} & =\mathrm{Tor}^{H_*(\THH(\Ss[x_n]);\Ff)}(\Ff,V(n)_*\THH(\BP{n}))\\
    & \cong E(\lambda_1,\cdots,\lambda_{n+1})\otimes P(\mu_{n+1})\otimes E([x_n])\otimes\Gamma([dx_n]),
\end{split} 
\end{equation*}
with the bidegrees of the generators as stated in the corollary. 

For degree reasons, the nonzero differentials are determined by the differentials on $\gamma_k([dx_n])$ for $k\geq p$, together with the Leibniz rule. Note that the bidegree of $d^r(\gamma_k([dx_n])$ is $(k-r,\ 2kp^n-k+r-1)$, for $r\geq2$ and $k\geq p$. In total degree $2p^{n+1}-1$, the $E^2$-page is generated by the three classes $\lambda_{n+1}$, $\lambda_n\cdot\gamma_{p-1}([dx_n])$ and $[x_n]\cdot\gamma_{p-1}([dx_n])$ whose bidegrees are $(0,\ 2p^{n+1}-1)$, $(p-1,\ 2p^{n+1}-p)$ and $(p,\ 2p^{n+1}-p-1)$, respectively. Since the dimension of $\mathrm{V}(n)_{2p^{n+1}-1}\THH(\BP{n-1})$ is $2$, one of these classes must be a boundary. If $\lambda_n\cdot\gamma_{p-1}([dx_n])$ were hit by a differential, then we would have
\[(p-1,\ 2p^{n+1}-p)=(k-r,\ 2kp^n-k+r-1)\]
for some $r\geq2$ and $k\geq p$. This forces $r=1$, a contradiction. Hence $\lambda_n\cdot\gamma_{p-1}([dx_n])$ cannot be hit by any differential. The same argument applies to $[x_n]\cdot\gamma_{p-1}([dx_n])$. Therefore, we deduce that
\[d^p(\gamma_p(x_n))\doteq\lambda_{n+1}.\]
Therefore, for each $1\leq i\leq n$, $[x_n]$, $[dx_n]$, $\mu_{n+1}$ and $\lambda_i$ are all permanent cycles and represent $\epsilon_n$ (modulo $\lambda_n$), $\mu_n$, $\mu_n^p$ and $\lambda_i$, respectively. As the spectral sequence is multiplicative, each generating monomial in $E(\lambda_1,\cdots,\lambda_n)\otimes P(\mu_{n+1})\otimes E([x_n]) \otimes P_p([dx_n])$ is also a permanent cycle which represents a non-zero product in $V(n)_*\THH(\BP{n-1})$. 

Again, for degree reasons, $\gamma_{k}([dx_n])$ cannot be a cycle for $k> p$. Hence the classes $\gamma_{k}([dx_n])$ are forced to be boundaries, which in turn kill the classes $\lambda_{n+1}\cdot\gamma_{\ell}([dx_n])$ for $\ell\geq 1$, by the discussion in the preceding paragraph. Now assume that $\lambda_{n+1}\cdot\gamma_{\ell}([dx_n])\doteq d^r(\gamma_k([dx_n])$. Then
\[(\ell,\ 2p^{n+1}-1+2\ell p^n-\ell)=(k-r,\ 2kp^n-k+r-1).\]
Hence $(k,r)=(\ell+p,p)$, and therefore $d^p(\gamma_{\ell+p}[dx_n])=\lambda_{n+1}\cdot\gamma_{\ell}([dx_n])$ for $\ell\geq 1$. 

The multiplicative extension is resolved by the above computation together with the multiplicative structure of $V(n)_*\THH(\BP{n-1}).$
\end{proof}

Now, we can prove our first main theorem.

\begin{thm}\label{thm:main}
Suppose that the Smith--Toda complex $V(n)$ exists as a ring spectrum. Then there is an isomorphism of $\Ff$-algebras
\[V(n)_*\THH(\BP{n},\langle v_n\rangle)\cong E(\lambda_1,\cdots,\lambda_n,d\mathrm{log}x_n)\otimes P(\kappa_n)\]
with $|dlogx_n|=1,$ $|\kappa_n|=2p^n$, and $|\lambda_i|=2p^i-1$ for $1\leq i\leq n$.
\end{thm}
\begin{proof}
By Lemma \ref{lem:ass}, Equation \eqref{eq:cospan} induces a cospan of Tor spectral sequences whose $E^2$-pages are
\begin{equation*}
\begin{tikzcd}[row sep=1.3em]
{\mathrm{Tor}^{E(dx_n)\otimes P(x_n)}_{*,*}(E(d\log x_n)\otimes P(x_n),E(\lambda_1,\cdots,\lambda_{n+1})\otimes P(\mu_{n+1}))}
\arrow[d]
\\
{\mathrm{Tor}^{E(dx_n)\otimes P(x_n)}_{*,*}(E(d\log x_n),E(\lambda_1,\cdots,\lambda_{n+1})\otimes P(\mu_{n+1}))}
\\
{\mathrm{Tor}^{E(dx_n)\otimes P(x_n)}_{*,*}(\Ff,E(\lambda_1,\cdots,\lambda_{n+1})\otimes P(\mu_{n+1}))}.
\arrow[u]
\end{tikzcd}
\end{equation*}
Here, as before, we use the same notation for a class and its image in the corresponding Tor spectral sequence. Standard Tor calculations imply that this cospan is isomorphic to
\begin{equation*}
\begin{tikzcd}[row sep=1.3em]
{E(\lambda_1,\cdots,\lambda_{n+1})\otimes P(\mu_{n+1})\otimes E(d\mathrm{log}x_n)\otimes \Gamma([dx_n])}
\arrow[d]
\\
{E(\lambda_1,\cdots,\lambda_{n+1})\otimes P(\mu_{n+1})\otimes E(d\mathrm{log}x_n)\otimes E([x_n])\otimes \Gamma([dx_n])}
\\
{ E(\lambda_1,\cdots,\lambda_{n+1})\otimes P(\mu_{n+1})\otimes E([x_n])\otimes \Gamma([dx_n])}.
\arrow[u]
\end{tikzcd}
\end{equation*}
In particular, both maps are injective. For degree reasons, the class $d\mathrm{log}x_n$ survives to the $E^p$-page. Hence, by naturality and injectivity, the $d^p$-differentials in the bottom spectral sequence also occur in the middle spectral sequence. Since the upper map is also injective, the same $d^p$-differential also occurs on the top spectral sequence, leaving
\[E^{p+1}_{*,*}(v_n)\cong E(\lambda_1,\cdots,\lambda_n)\otimes P(\mu_{n+1})\otimes E(d\mathrm{log}x_n)\otimes P_p([dx_n]),\]
by Corollary \ref{cor:bpn-1}.
Since all surviving classes are concentrated in filtration degree at most $p-1$, there can be no further differentials. Hence, $E^{\infty}_{*,*}(v_n)\cong E^{p+1}_{*,*}(v_n)$. By Corollary \ref{cor:bpn-1}, there is a multiplicative extension $[dx_n]^p\doteq\mu_{n+1}$ in the bottom spectral sequence. Hence we deduce that 
\[V(n)_*\THH(\BP{n},\langle v_n\rangle)\cong E(\lambda_1,\cdots,\lambda_n,d\mathrm{log}x_n)\otimes P(\kappa_n),\]
with $\kappa_n$ represented by $[dx_n]$.
\end{proof}

\subsection{The multi-generator prelog structure \texorpdfstring{$\VV{n}$}{V(n)}}

We now turn to the prelog structure generated by $\langle v_1,\cdots,v_n\rangle$. As explained in the introduction, we compute its logarithmic THH by iterated logarithmic base change, repleting one generator at a time. For this purpose, for $1\leq i\leq n$, note that the $\Ee_2$-ring map
\[\Ss[x_1,\cdots,x_i]\xrightarrow{\otimes_{1\leq j\leq i}\tilde{v}_j}\otimes_{1\leq j\leq i} \BP{n}\xrightarrow{\mu}\BP{n}\]
induces a prelog $\Ee_2$-ring structure on $\BP{n}$, which is denoted by $(\BP{n},\VV{i})$. For $1\leq i\leq n$, we write
\[A_i:=\Ff\otimes\THH(\Ss[x_i]),\qquad B_i:=\Ff\otimes\THH(\Ss[x_i^{\pm1}])_{\geq0}.\]
We also set $L_{n,0}:=\Ff\otimes_{\BP{n}}\THH(\BP{n})$. For $1\leq i\leq n$, the $A_i$-action on $L_{n,0}$ is preserved under base change along
\[A_1\otimes\cdots\otimes A_{i-1}\to B_1\otimes\cdots\otimes B_{i-1}.\]
Hence $L_{n,i-1}$ is an $A_i$-module, and we define
\[L_{n,i}:=L_{n,i-1}\otimes_{A_i}B_i,\]
inductively. By Lemma \ref{lem:gd}, Lemma \ref{lem:ass}, and induction on $i$, there is an equivalence
\[L_{n,i}\simeq \Ff\otimes_{\BP{n}}\THH(\BP{n},\VV{i})\simeq V(n)\otimes \THH(\BP{n},\VV{i}) ,\]
for $1\leq i\leq n$. Our second main theorem is an immediate corollary of the following proposition.

\begin{thm}\label{thm:anyn}
	Suppose that the Smith--Toda complex $V(n)$ exists as a ring spectrum. Then there is an isomorphism of $\Ff$-modules
	\[\pi_*L_{n,i}\cong E(\lambda_1,\lambda_{i+2},\lambda_{n+1},d\mathrm{log}x_1,\cdots,d\mathrm{log}x_i)\otimes P(\mu_{n+1})\otimes\bigotimes_{r=1}^iP_p([dx_r]),\]
	for $1\leq i\leq n$.
\end{thm}

For $1\leq i\leq n$, we have the following Tor spectral sequence
\begin{equation}
	E_{*,*}^2(n,i)\cong\mathrm{Tor}_{*,*}^{E(dx_i)\otimes P(x_i)}(\pi_*L_{n,i-1},E(d\mathrm{log}x_i)\otimes P(x_i))\Longrightarrow \pi_*L_{n,i},
\end{equation}
by Proposition \ref{prop:S[x]} and the definition. Since the $A_i$-module structure on $L_{n,i-1}$ is induced by the $\Ee_1$-ring map
	\[\Ff\otimes\THH(\Ss[x_i])\to \Ff\otimes\THH(\BP{n})\to\Ff\otimes_{\BP{n}}\THH(\BP{n})\simeq V(n)\otimes\THH(\BP{n}),\]
 the same argument as in the proof of Corollary \ref{cor:bpn-1} implies that the elements $x_i$ and $dx_i$ act trivially on $\pi_*L_{n,i-1}$. Consequently, there is an isomorphism
 \begin{equation}\label{eq:induction}
 	E_{*,*}^2(n,i)\cong \pi_*L_{n,i-1}\otimes E(d\mathrm{log}x_i)\otimes\Gamma([dx_i]).
\end{equation}
Theorem \ref{thm:anyn} is a consequence of the following proposition.
\begin{prop}\label{prop:ssind}
	The $E^2$-page of $E^r(n,i)$ is isomorphism to
	\[E(\lambda_1,\lambda_{i+1},\cdots,\lambda_{n+1},d\mathrm{log}x_1,\cdots,d\mathrm{log}x_{i-1})\otimes P(\mu_{n+1})\otimes\bigotimes_{r=1}^{i-1}P_p([dx_r])\otimes E(d\mathrm{log}x_i)\otimes\Gamma([dx_i]).\]
	The nonzero differentials in $E^r(n,i)$ are generated by
    \begin{equation}\label{eq:dp}d^p(\gamma_k(dx_i)c)\doteq\lambda_{i+1}\gamma_{k-p}([dx_i])c, \qquad k\geq p
    \end{equation}
    for $c\in C_i$, together with linearity over $V(n)_*(\THH(\BP{n})$. Here 
    \[c\in C_i=\left\{\prod_{r=1}^{i-1}(d\mathrm{log}x_r)^{\epsilon_r}[dx_r]^{a_r} \ | \  \epsilon_r\in\{0,1\},\ 0\leq a_r<p\right\}.\] 
\end{prop}

\begin{proof}[proof of Theorem \ref{thm:anyn}]
	By Proposition \ref{prop:ssind}, there is an isomorphism
	\[E^{p+1}_{*,*}(n,i)\cong E(\lambda_1,\lambda_{i+2},\cdots,\lambda_{n+1},d\mathrm{log}x_1,\cdots,d\mathrm{log}x_i)\otimes P(\mu_{n+1})\otimes\bigotimes_{r=1}^iP_p([dx_r]),\]
	which is concentrated in filtration degrees at most $p-1$. Hence no differential $d^r$ with $r>p$ can occur, and therefore $E^{p+1}_{*,*}(n,i)\cong E^\infty_{*,*}(n,i).$
\end{proof}

We prove this proposition by induction on both $n$ and $i$. For this, we need the following lemmas.

\begin{lem}
	The spectral sequence $E^r_{*,*}(n,i)$ is a spectral sequence of $V(n)_*\THH(\BP{n})$-modules.
\end{lem}
\begin{proof}
	Since $\Ff\otimes_{\BP{n}}\THH(\BP{n})$ is an $\Ee_2$-ring, the $\Ee_1$-ring map
	\begin{equation*}
	\begin{split}
		\bigotimes_{r=1}^iA_r\otimes (\Ff\otimes_{\BP{n}}\THH(\BP{n})) & 
		\to (\Ff\otimes_{\BP{n}}\THH(\BP{n}))\otimes (\Ff\otimes_{\BP{n}}\THH(\BP{n}))\\
		& \xrightarrow{\mu} (\Ff\otimes_{\BP{n}}\THH(\BP{n}))
	\end{split}
	\end{equation*}
	equips $\Ff\otimes_{\BP{n}}\THH(\BP{n})$ with a right $(\bigotimes_{r=1}^iA_r\otimes (\Ff\otimes_{\BP{n}}\THH(\BP{n})))$-module structure. In other words, the natural $\Ff\otimes_{\BP{n}}\THH(\BP{n})$-module structure on $\Ff\otimes_{\BP{n}}\THH(\BP{n})$ is compatible with the right $(\bigotimes_{r=1}^iA_r)$-module structure on $\Ff\otimes_{\BP{n}}\THH(\BP{n})$.
	Consequently, the bar construction $B(\Ff\otimes_{\BP{n}}\THH(\BP{n}),\bigotimes_{r=1}^iA_r,\bigotimes_{r=1}^iB_r)$ can be formed in the category of $(\Ff\otimes_{\BP{n}}\THH(\BP{n})$-modules, which implies the desired result.
\end{proof}

Given $1\leq i< n$, assume that Proposition \ref{prop:ssind} is true for $j<i$. Then Equation \eqref{eq:induction} implies that the $E^2$-page of the spectral sequence $E_{*,*}^r(n,i)$ is isomorphic to
\begin{equation}\label{eq:ind,i}
	E(\lambda_1,\lambda_{i+1},\cdots,\lambda_{n+1},d\mathrm{log}x_1,\cdots,d\mathrm{log}x_{i-1})\otimes P(\mu_{n+1})\otimes\bigotimes_{r=1}^{i-1}P_p([dx_r])\otimes E(d\mathrm{log}x_i)\otimes\Gamma([dx_i]).
\end{equation}
\begin{lem}\label{lem:fac}
	For $1\leq i< n$, there is an equivalence of $\Ff\otimes_{\BP{n}}\THH(\BP{n})$-modules
	\[\widetilde{L}_{n-1,i}:=(\Ff\otimes_{\BP{n}}\THH(\BP{n-1}))\otimes_{\bigotimes_{r=1}^iA_r}{\bigotimes_{r=1}^iB_r}\simeq L_{n-1,i}\oplus\Sigma^{2p^n-1}L_{n-1,i}.\]
	Moreover, the $\Ee_3$-ring map $f_n\colon\BP{n}\to\BP{n-1}$ induces a map of $\Ff\otimes_{\BP{n}}\THH(\BP{n})$-modules $f_{n,i}\colon L_{n,i}\to \tilde{L}_{n-1,i}$ which factors as
	\[L_{n,i}\xrightarrow{\tilde{f}_{n,i}} L_{n-1,i}\hookrightarrow \widetilde{L}_{n-1,i}.\]
\end{lem}
\begin{proof}
	For the first assertion, it follows from a similar argument as the above lemma that we have the following equivalences of right $(\Ff\otimes_{\BP{n}}\THH(\BP{n}))\otimes\bigotimes_{r=1}^iA_r$-modules
	\begin{equation*}
		\begin{split}
			\Ff\otimes_{\BP{n}}\THH(\BP{n-1}) & \simeq \Ff\otimes_{\BP{n}}(\BP{n-1}\otimes_{\BP{n-1}}\THH(\BP{n-1}))\\
			& \simeq (\Ff\otimes_{\BP{n}}\BP{n-1})\otimes_{\BP{n-1}}\THH(\BP{n-1})\\
			& \simeq (\Ff\oplus\Sigma^{2p^n-1}\Ff)\otimes_{\BP{n-1}}\THH(\BP{n-1}).
		\end{split}
	\end{equation*}
	since the map $f_n$ is an $\Ee_3$-ring map. To see the second assertion, note that, under the above equivalence, Lemma \ref{lem:lift} implies that the map  $L_{n,i}\to \tilde{L}_{n-1,i}$ is given by the bar construction
	\begin{equation*}
		\begin{split}
			& B(\Ff\otimes_{\BP{n}}\BP{n},\BP{n},\THH(\BP{n}))\otimes_{\bigotimes_{r=1}^iA_r}\bigotimes_{r=1}^iB_r)\\ 
			& \to B(\Ff\otimes_{\BP{n}}\BP{n-1},\BP{n-1},\THH(\BP{n-1}))\otimes_{\bigotimes_{r=1}^iA_r}\bigotimes_{r=1}^iB_r),
		\end{split}
	\end{equation*}
	and that $\Ff\otimes_{\BP{n}}\BP{n}\to \Ff\otimes_{\BP{n}}\BP{n-1}$ is equivalent to the canonical inclusion $\Ff\to\Ff\otimes\Sigma^{2p^n-1}\Ff$.
\end{proof}

\begin{lem}\label{lem:inj}
	Let $1\leq i<n$. Assume that Proposition \ref{prop:ssind} is true at height $n-1$ and for $j<i$. Then $\tilde{f}_{n,i}$ induces a map of $V(n)_*\THH(\BP{n})$-module spectral sequences $\tilde{f}_{n,i}^{\ 2}\colon E^2_{n,i}\to E^2_{n-1,i},$ which is injective on the submodule not involving $\lambda_{n+1}$ and $\mu_{n+1}$. In particular, $\tilde{f}_{n,i}^{\ 2}$ is injective on every class represented by such a monomial whose total degree is strictly smaller than $2p^{n+1}-1$.
\end{lem}
\begin{proof}
	Consider the following commutative diagram
	\[
	\begin{tikzcd}
	V(n)_*\THH(\BP{n}) \arrow[r] \arrow[d] & \pi_*L_{n,i} \arrow[d] \\
	V(n)_*\THH(\BP{n-1}) \arrow[r] & \pi_*L_{n-1,i}.
	\end{tikzcd}
	\]
	By the construction of the spectral sequences, $\lambda_1,\cdots,\lambda_n\in E^2_{*,0}(n,i)$ come from the image of the top map, and the classes with the same names in $E^2_{*,0}(n-1,i)$ come from the image of the bottom map. Hence $\tilde{f}_{n,i}^{\ 2}$ maps $\lambda_1,\cdots,\lambda_n$ to the classes with the same names. Note that, for $1\leq j <i$, $\tilde{f}_{n,j}$ is obtained by the bar construction
	\[B(L_{n,j-1},A_j,B_j)\to B(L_{n-1,j-1},A_j,B_j).\]
	Hence an induction on $j$ and the construction of the spectral sequences imply that $\tilde{f}_{n,i}^{\ 2}$ maps $d\mathrm{log}x_1,\cdots,d\mathrm{log}x_i$ and $[dx_1],\cdots,[dx_i]$ to the classes with the same names. The result follows from Equation \eqref{eq:ind,i}.
\end{proof}

\begin{proof}[proof of Proposition \ref{prop:ssind}]
	The case $n=1$ follows from Theorem \ref{thm:main}. Now, for a fixed $n$ and $1<i<n$, suppose that the proposition is true for $k<n$ and $j<i$ for this $n$, we prove that the proposition is true for $n$ and $j=i$. First, the $E^2$-page of the spectral sequence has been identified by Equation \eqref{eq:bpn-1} and a computation as in the proof of Theorem \ref{thm:anyn}.
	
	Next, assume that $1\leq i<n$. Lemma \ref{lem:lift} and Lemma \ref{lem:fac} implies that there is a comparison map of spectral sequences $\tilde{f}^{\ r}_{n,i}\colon E^r(n,i)\to E^r(n-1,i)$ which converges to the map $\pi_*L_{n,i}\to \pi_*L_{n-1,i}$. By the inductive hypothesis on $n$, there are nonzero differentials in the latter spectral sequence 
	\[d^p(c\cdot\gamma_k(dx_i))\doteq\lambda_{i+1}\cdot c\cdot\gamma_{k-p}([dx_i]),\]
	with $c\in C_i$, and no earlier nonzero differential. Hence it remains to show that no differential in $E^r(n,i)$ can differ from the differential in $E^r(n-1,i)$ by a class in the kernel of the comparison map. For $c\in C_i$, a differential $d^r(c\cdot\gamma_k([dx_i])$ has the filtration degree $k-r$. Since $[dx_i]$ is the only positive filtration generator in $E^r(n,i)$, the possible target of this differential has the form $c^{\prime}\cdot\gamma_{k-r}([dx_i])$, which implies that
	\[|c^{\prime}|=|c|+2rp^i-1.\]
	Hence, for $2\leq r\leq p$, we have that
	\[|c^{\prime}|\leq \Sigma_{j=1}^{i-1}2p^r+i-1+2p^{i+1}-1\leq 2p^{i+1}+2p^i-2p+i-2.\]
	Since $n\geq i+1$ and $p$ is odd, we deduce that 
	\[|c^{\prime}|<2p^{n+1}-1.\]
	Therefore, Lemma \ref{lem:inj} implies that $\tilde{f}^{\ r}_{n,i}$ is injective on every possible target of a differential with $r\leq p$. For $2\leq r<p$, since the differential in $E^r_{*,*}(n-1,i)$ is trivial, so is that in $E^r_{*,*}(n,i)$. For $r=p$, naturality and injectivity force $c^{\prime}$ must be $\lambda_{i+1}c$, up to a nonzero scalar in $\Ff$. The desired result follows from that $\tilde{f}^{\ r}_{n,i}$ is $V(n)_*\THH(\BP{n})$-linear, by Lemma \ref{lem:inj}.
	
	It remains to consider the case $i=n$. Note that there is an $\Ff\otimes\THH(\mathrm{MU}[x_n])$-action on $L_{n,n-1}$ is induced by the $\Ff\otimes\THH(\mathrm{MU}[x_n])$-action on $L_{n,0}$ which is preserved under the base change along $\bigotimes_{i=1}^{n-1}A_i\to \bigotimes_{i=1}^{n-1}B_i$. Moreover, there is a natural map 
	\[L_{n,n}=B(L_{n,n-1},A_i,B_i)\simeq B(L_{n,n-1},\Ff\otimes\THH(\mathrm{MU}[x_n]),\Ff\otimes\THH(\mathrm{MU}[x_n^{\pm1}])).\]
	Since the bar construction is symmetric monoidal in the variables $(M,A,N)$, a similar argument as the proof of Proposition \ref{prop:n-1case} implies that $L_{n,n}$ is a $\Ff\otimes_{\BP{n}}\THH(\BP{n},\langle v_n\rangle)$-module, and the associated spectral sequence $E^r_{*,*}(n,n)$ is a module spectral sequence over the spectral sequence $E^r_{*,*}(v_n)$. Consequently, we get the desired differentials, by Theorem \ref{thm:main}.
\end{proof}

\begin{cor}
	Let $p\geq7$. There is an isomorphism 
	\[V(2)_*\THH(\BP{2},\langle v_1\rangle)\cong E(\lambda_1,\lambda_3,d\mathrm{log}x_1)\otimes P(\mu_3)\otimes P_p([dx_1]).\]
\end{cor}

More generally, let $J=\{j_1<\cdots<j_k\}$ be a subset of $\{1,\cdots,n\}$. The same argument, applied successively to the generators indexed by $J$, gives the following result.

\begin{thm}
	Suppose that the Smith--Toda complex $V(n)$ exists as a ring spectrum. Then, for $1\leq k\leq n$ there is an isomorphism of $\Ff$-modules
	\begin{equation*}
		\begin{split}
			& V(n)_*\THH(\BP{n},\langle v_j,j\in J\rangle)\\
			&\cong E(\lambda_r,\ d\mathrm{log}x_j|\ r\in \{1,\cdots,n+1\}\setminus\{j_1+1,\cdots,j_k+1\},\ j\in J)\otimes P(\mu_{n+1})\otimes \bigotimes_{j\in J}P_p([dx_j]),
		\end{split}
	\end{equation*}
	where $|d\mathrm{log}x_j|=1,|\mu_{n+1}|=2p^{n+1},|[dx_j]|=2p^j$, and $|\lambda_r|=2p^r-1$.
\end{thm}
\begin{proof}
	Apply the proof of Proposition \ref{prop:ssind} successively to $j_1,\cdots,j_k$. At the $j_l$-th step the coefficient monomials are indexed by a subset of $\{1,\cdots,j_l-1\}$, so the same degree estimate and differential argument apply verbatim.
\end{proof}

\section{Log THH of \texorpdfstring{$\BP{2}/v_1$}{BP<2>/v1}}\label{sec5}
In this section, we consider logarithmic topological Hochschild homology of $\BP{2}/v_1$. To this end, we first define a prelog structure on $\BP{2}/v_1$. Since $\BP{2}/v_1$ is only an $\Ee_1$-ring, the $\Ee_2$ even-cell decomposition of $\Ss[x]$ does not directly give rise to a map $\Ss[x]\to \BP{2}/x$. Instead, we follow the method of \cite[\S 6]{ABM22}.

Using the $\Ee_2$-augmentation $\Ss[x_1]\to \Ss$, we can identify $\mathbb{S}$ as an object of $\mathrm{LMod}_{\Ss[x_1]}(\mathrm{Alg}_{\Ee_1}(\mathrm{Sp}))$. Since the map $\bar{v}_1\colon \Ss[x_1]\to \BP{2}$ is a map of $\Ee_2$-rings, base change along $\bar{v}_1$ implies that
\[\BP{2}/v_1\simeq \BP{2}\otimes_{\Ss[x_1]}\Ss\in \mathrm{LMod}_{\BP{2}}(\mathrm{Alg}_{E_1}(\mathrm{Sp})).\]
In particular, the canonical map $\BP{2}\to\BP{2}/v_1$ is an $\Ee_1$-ring map. Precomposing with $\bar{v}_2$ gives a map $\bar{v}_2\colon\Ss[x_2]\to\BP{2}/v_1$, which gives a prelog structure on $\BP{2}/v_1$. Moreover, restriction of scalar along $\bar{v}_2$ shows that $\BP{2}/v_1\in \mathrm{LMod}_{\Ss[x_2]}(\mathrm{Alg}_{E_1}(\mathrm{Sp}))$. Hence $\THH(\BP{2}/v_1)$ is a $\THH(\Ss[x_2])$-module.
\begin{Def}\cite[Definition 6.6]{ABM22}
	The logarithmic topological Hochschild homology of $\BP{2}/v_1$ relative to $v_2$ is defined to be 
	\[\THH(\BP{2}/v_1,\langle v_2\rangle):=\THH(\BP{2}/v_1)\otimes_{\THH(\Ss[x_2])}\THH(\Ss[x_2^{\pm1}])_{\geq 0}.\]
\end{Def}
\begin{rem}
	It follows from the definition that there is a cofiber sequence
	\[\THH(\BP{2}/v_1)\to \THH(\BP{2}/v_1,\langle v_2\rangle)\to \Sigma\THH(\Zz_{(p)}).\]
	Indeed, this cofiber sequence can be obtained by tensoring the following cofiber sequence
	\[\THH(\BP{2})\to \THH(\BP{2},\langle v_2\rangle)\to \Sigma\THH(\BP{1})\]
	with $\THH(\BP{2}/v_1)$ over $\THH(\BP{2})$.
\end{rem}

The goal of this section is to computing the $V(2)$-homotopy groups of $\THH(\BP{2}/v_1,\langle v_2\rangle)$. We begin by compute the $V(2)$-homotopy groups of $\THH(\BP{2}/v_1)$. 
\begin{prop}\label{cor:bp/v}
	Let $p\geq 7$. There is an isomorphism 
	\[\THH_*(\BP{2}/v_1;\Ff)\cong E(\lambda_1,\lambda_3)\otimes P(\mu_3)\otimes P_p([dx_1]).\]
\end{prop}
\begin{proof}
	By Lemma \ref{lem:ass}, there is an equivalence
	\[(\Ff\otimes_{\BP{n}}\THH(\BP{2}))\otimes_{\Ff\otimes\THH(\Ss[x_1])}\Ff\simeq \Ff\otimes_{\BP{n}}\THH(\BP{2}/v_1)\simeq V(2)_*\THH(\BP{2}).\]
	Hence there is a spectral sequence 
	\[E^2_{*,*}(\BP{2}/v_1)\cong \mathrm{Tor}^{E(dx_1)\otimes P(x_1)}(E(\lambda_1,\lambda_2,\lambda_3)\otimes P(\mu_3),\Ff)\Longrightarrow V(2)_*\THH(\BP{2}/v_1),\]
	whose $E^2$-page is isomorphic to $E(\lambda_1,\lambda_2,\lambda_3)\otimes P(\mu_3)\otimes E([x_1])\otimes\Gamma([dx_1]),$ by a similar argument as the proof of Corollary \ref{cor:bpn-1}. The same reason as that in the proof of Proposition \ref{prop:ssind} applied to the map
	\[(\Ff\otimes_{\BP{2}}\THH(\BP{2}))\otimes_{\Ff\otimes\THH(\Ss[x_1])}\Ff\to \Ff\otimes_{\BP{2}}\THH(\BP{1}))\otimes_{\Ff\otimes\THH(\Ss[x_1])}\Ff \]
	implies that the only nonzero differentials are given by $d^p(\gamma_k([dx_1]c)=\lambda_2\gamma_{k-p}([dx_1])c$, for $k\geq p$ and $c\in E(\lambda_1,\lambda_2,\lambda_3)\otimes P(\mu_3)\otimes E([x_1]).$ Consequently, we deduce that
	\[V(2)_*\THH(\BP{2}/v_1)\cong E(\lambda_1,\lambda_3,\epsilon_1)\otimes P(\mu_3)\otimes P_p([dx_1]).\]
	Since $V(2)\otimes\BP{2}/v_1$ splits as a $\BP{2}/v_1$-bimodule into $\Ff\oplus\Sigma^{2p-1}\Ff$, we deduce that there is an equivalence of $\THH(\BP{2};\Ff)$-modules
	\begin{equation}\label{eq:split}
	\begin{split}
		V(2)\otimes\THH(\BP{2}/v_1) & \simeq (V(2)\otimes\BP{2}/v_1)\otimes_{\BP{2}/v_1\otimes(\BP{2}/v_1)^{\mathrm{op}}} \BP{2}/v_1\\
		& \simeq \THH(\BP{2}/v_1;\Ff)\oplus\Sigma^{2p-1}\THH(\BP{2}/v_1;\Ff).
	\end{split}
	\end{equation}
	On the other hand, the canonical map $V(2)\otimes\THH(\BP{2})\to V(2)\otimes\THH(\BP{2}/v_1)$ factors as
	\begin{equation*}
		\begin{split}
			& V(2)\otimes\THH(\BP{2}) \simeq (V(2)\otimes\BP{2})\otimes_{\BP{2}\otimes\BP{2}^{\mathrm{op}}} \BP{2} \\
			& \longrightarrow (V(2)\otimes\BP{2}/v_1)\otimes_{\BP{2}/v_1\otimes(\BP{2}/v_1)^{\mathrm{op}}} \BP{2}/v_1\simeq V(2)\otimes\THH(\BP{2}/v_1).
		\end{split}
	\end{equation*}
	Tensoring the cofiber sequence $\Sigma^{2p-2}\BP{2}\xrightarrow{v_1}\BP{2}\to\BP{2}/v_1$ with $V(2)$, we obtain a cofiber sequence 
	\[\Sigma^{2p-2}(V(2)\otimes\BP{2})\xrightarrow{v_1} V(2)\otimes\BP{2}\to V(2)\otimes\BP{2}/v_1.\]
	Since $V(2)\otimes\BP{2}\simeq \Ff$ and $v_1$ acts trivially after tensoring with $V(2)$, the first map is null. Hence the above cofiber sequence can be identified with
	\[\Sigma^{2p-2}\Ff\xrightarrow{0} \Ff\xrightarrow{\mathrm{id}_{\Ff}\oplus\  0}\Ff\oplus\Sigma^{2p-1}\Ff.\]
	Together with the equivalence \eqref{eq:split}, we see that the canonical map $V(2)\otimes\THH(\BP{2})\to V(2)\otimes\THH(\BP{2}/v_1)$ factors as
	\[V(2)\otimes\THH(\BP{2})\to \THH(\BP{2}/v_1;\Ff)\to V(2)\otimes\THH(\BP{2}/v_1).\]
	Therefore, $\THH_*(\BP{2}/v_1;\Ff)$ contains the image of $V(2)_*\THH(\BP{2})\simeq \THH_*(\BP{n};\Ff)$. Since the second map above is a map of $\THH(\BP{2};\Ff)$-modules, the desired result follows from the equivalence \eqref{eq:split}.
\end{proof}

Since $\BP{2}/v_1\in \mathrm{LMod}_{\BP{2}}(\mathrm{Alg}_{E_1}(\mathrm{Sp}))$ and $\BP{2}$ is an $\Ee_3$-ring, the same argument as in the proof of Proposition \ref{prop:eqV(n)} implies that 
\begin{equation}\label{eq:bp/v}
	\Ff\otimes_{\Ff\otimes\THH(\Ss[x_2])}\Ff\otimes_{\BP{2}}\THH(\BP{2}/v_1)\simeq V(2)\otimes\THH(\Zz_{(p)}).
\end{equation}
Again, to compute the spectral sequences, we need a module structure on this spectrum.

\begin{lem}\label{lem:mod3}
	The spectrum
	\[\Ff\otimes_{\Ff\otimes\THH(\Ss[x_2])}\Ff\otimes_{\BP{2}}\THH(\BP{2}/v_1)\]
	is a left $\Ff\otimes_{\Ff\otimes\THH(\Ss[x_2])}\Ff\otimes_{\BP{2}}\THH(\BP{2})$-module.
\end{lem}
\begin{proof}
	Consider the span of $\Ee_2$-rings
	\[\Ff\otimes\THH(\mathrm{MU})\xleftarrow{\eta}\Ff\otimes\THH(\mathrm{MU}[x_2])\xrightarrow{\bar{v}_2} \Ff\otimes_{\BP{2}}\THH(\BP{2}).\]
	Since $\Ff\otimes_{\BP{2}}\THH(\BP{2}/v_1)$ is a left $\Ff\otimes_{\BP{2}}\THH(\BP{2})$-module and $\bar{v}_2$ is a map of $\Ee_2$-rings, we can consider $\Ff\otimes_{\BP{2}}\THH(\BP{2}/v_1)$ as a left $\Ff\otimes_{\BP{2}}\THH(\BP{2})$-module in the category $\mathrm{LMod}_{\Ff\otimes\THH(\mathrm{MU}[x_2])}(\mathrm{Sp})$, by the equivalence (see \cite[Corollary 3.4.1.9]{HA})
	\[\mathrm{LMod}_{\Ff\otimes_{\BP{2}}\THH(\BP{2})}(\mathrm{Sp})\simeq \mathrm{LMod}_{\Ff\otimes_{\BP{2}}\THH(\BP{2})}(\mathrm{LMod}_{\Ff\otimes\THH(\mathrm{MU}[x_2])}(\mathrm{Sp})).\]
	Base change along $\eta$ implies that $\Ff\otimes_{\Ff\otimes\THH(\Ss[x_2])}\Ff\otimes_{\BP{2}}\THH(\BP{2}/v_1)$ is a left $\Ff\otimes_{\Ff\otimes\THH(\Ss[x_2])}\Ff\otimes_{\BP{2}}\THH(\BP{2})$-module.
\end{proof}

\begin{thm}
	Let $p\geq7$. There is an isomorphism
	\[V(2)_*\THH(\BP{2}/v_1,\langle v_2\rangle)\cong E(\lambda_1,\epsilon_1,d\mathrm{log}x_2)\otimes P(\mu_3)\otimes P_p([dx_1])\otimes P_p([dx_2]).\]
\end{thm}
\begin{proof}
	The computation proceeds in the same way as in Section \ref{sec4}. First, by the equivalence \eqref{eq:bp/v}, there is a Tor spectral sequence $E^r_{*,*}(\Zz_{(p)})$ computing $V(2)_*\THH(\Zz_{(p)})$ whose $E^2$-page is given by
	\begin{equation*}
	\begin{split}
		E^2_{*,*}(\Zz_{(p)}) & \cong\mathrm{Tor}^{H_*(\THH(\Ss[x_2]);\Ff)}(\Ff,V(2)_*\THH(\BP{2}/v_1))\\
		& \cong E(\lambda_1,\lambda_3,\epsilon_1)\otimes P(\mu_3)\otimes P_p([dx_1])\otimes E([x_2])\otimes\Gamma([dx_2]).
	\end{split}
	\end{equation*}
	Using the equivalence \eqref{eq:split}, we see that this spectral sequence decomposes as the direct sum of two spectral sequences, which are denoted by $E^r_{*,*}(0)$ and $E^r_{*,*}(1)$, respectively.	Moreover, the canonical map $\BP{2}\to\BP{2}/v_1$ induces a map of spectral sequences $E^r_{*,*}(\BP{1})\to E^r_{*,*}(\Zz_{(p)})$, which factors through $E^r_{*,*}(0)$ by the proof of Proposition \ref{cor:bp/v}. On the $E^2$-page, we can identify the map $E^2_{*,*}(\BP{1})\to E^2_{*,*}(0)$ with
	\[E(\lambda_1,\lambda_2,\lambda_3)\otimes P(\mu_3)\otimes E([x_2])\otimes \Gamma([dx_2])\to E(\lambda_1,\lambda_3)\otimes P(\mu_3)\otimes P_p([dx_1])\otimes E([x_2])\otimes\Gamma([dx_2]),\]
	which sends $\lambda_2$ to zero and sends all the other classes to the corresponding classes denoted by the same symbols. For degree reasons, $\lambda_3$ survives to the $E^p$-page. By naturality and a degree argument, $\gamma_i([dx_2])$ also survives to the $E^p$-page for all $i\ge 0$. Hence we obtain that there are differentials
	\[d^p(\gamma_{i}([dx_2])=\lambda_3\cdot\gamma_{i-p}([dx_2]),\ \ i\geq p\]
	in $E^r_{*,*}(0)$ by Corollary \ref{cor:bpn-1}. By Lemma \ref{lem:mod3}, the map of spectral sequences $E^r_{*,*}(\BP{1})\to E^r_{*,*}(\Zz_{(p)})$ is $E^r_{*,*}(\BP{1})$-linear, and so is $E^r_{*,*}(\BP{1})\to E^r_{*,*}(0)$. A similar argument shows that there are differentials 
	\[d^p(\epsilon_1\cdot\gamma_i([dx_2]))=\epsilon_1\cdot\lambda_3\cdot\gamma_{i-p}([dx_2]),\ \ i\geq p\]
	in $E^r_{*,*}(1)$. Again, Lemma \ref{lem:mod3} implies that the differential in $E^r_{*,*}(1)$ is $E^r_{*,*}(\BP{1})$-linear. Therefore, after taking the above differentials and $E^r_{*,*}(\BP{1})$-linearity into account, the resulting page has the same dimension as $V(2)_*\THH(\Zz_{(p)})$ in each degree. Since the spectral sequence is degreewise finite dimensional and converges to $V(2)_*\THH(\Zz_{(p)})$, there is no room for any further nonzero differential. Hence the above differentials, together with $E^r_{*,*}(\BP{1})$-linearity, gives all nontrivial differentials in $E^r_{*,*}(\Zz_{(p)})$. The same argument as in the proof of Theorem \ref{thm:main} gives the desired result.
\end{proof}

\section{Maps in the repletion-residue cofiber sequences}\label{sec6}

Recall from \cite[Theorem 1.4]{RSS25} that there are repletion-residue cofiber sequences
\begin{equation}\label{seq1}
	\THH(\BP{n})\xrightarrow{\rho_n} \THH(\BP{n},\langle v_n\rangle)\xrightarrow{\mathrm{res}_n}\Sigma\THH(\BP{n-1}),
\end{equation}
and
\begin{equation}\label{seq2}
	\THH(\BP{2}/v_1)\xrightarrow{\bar{\rho}_2} \THH(\BP{2}/v_1,\langle v_2\rangle)\xrightarrow{\overline{\mathrm{res}}_2}\Sigma\THH(\Zz_{(p)}).
\end{equation}
The preceding sections determine the $V(n)$-homotopy groups of terms in these cofiber sequences. In this section, we determine the homomorphisms between them. We first treat the cofiber sequence associated to the single generator $v_n$. We then specialize to height two and compare it with the corresponding sequence for $\BP{2}/v_1$.

\subsection{The cofiber sequence associated to \texorpdfstring{$\langle v_n\rangle$}{<v_n>}}
For convenience, we denote $\mathrm{res}_n$ by $\partial_n$. Let
\[\tau_n\colon\THH(\BP{n-1})\longrightarrow\THH(\BP{n})\]
denote the fiber map of $\rho_n$. The maps $\rho_n$, $\partial_n$, and
$\tau_n$ are known as the localization, residue, and transfer maps, respectively. The following result is the analogue of \cite[Lemma 7.4]{RSS18}.

\begin{prop}\label{prop:seq1}
	Suppose that the Smith--Toda complex $V(n)$ exists as a ring spectrum. Then the long exact sequence
	\begin{multline*}
	\cdots\longrightarrow V(n)_*\THH(\BP{n-1}) \xrightarrow{(\tau_n)_*}V(n)_*\THH(\BP{n}) \xrightarrow{(\rho_n)_*} V(n)_*\THH(\BP{n},\langle v_n\rangle)\\
	\xrightarrow{(\partial_n)_*}V(n)_{*-1}\THH(\BP{n-1} \longrightarrow\cdots
\end{multline*}
    is a sequence of $V(n)_*\THH(\BP{n})$-modules, and $(\rho_n)_*$ is an algebra homomorphism. With the generators chosen in Theorem \ref{thm:main} and Equation \eqref{eq:bpn-1}, the localization map satisfies
    \[(\rho_n)_*(\lambda_i)=\lambda_i,\qquad (\rho_n)_*(\lambda_{n+1})=0,\qquad (\rho_n)_*(\mu_{n+1})\doteq\kappa_n^p,\]
    for $1\leq i\leq n$. The residue map satisfies
    \begin{alignat*}{2}
    (\partial_n)_*(\kappa_n^k)
    &= 0, &\qquad& p\mid k, \\
    (\partial_n)_*(\kappa_n^k)
    &\doteq \epsilon_n\mu_n^{k-1} \pmod{\lambda_n\mu_n^{k-1}}, && p\nmid k,\quad k\geq 1, \\
    (\partial_n)_*(d\log x_n\cdot \kappa_n^k)
    &\doteq \mu_n^k, && k\geq 0.
    \end{alignat*}
    Finally, the transfer map satisfies
    \begin{alignat*}{2}
    (\tau_n)_*(\mu_n^q)
    &= 0, &\qquad& q\geq 0, \\
    (\tau_n)_*(\epsilon_n\mu_n^q)
    &= 0, && q\not\equiv p-1 \pmod p, \\
    (\tau_n)_*(\epsilon_n\mu_n^{pr-1})
    &\doteq \lambda_{n+1}\mu_{n+1}^{r-1}, && r\geq 1.
    \end{alignat*}
    Together with $V(n)_*\THH(\BP{n})$-linearity, these formulas determine all three homomorphisms.
\end{prop}
\begin{proof}
   The map $(\rho_n)_*$ is induced by the canonical map from ordinary THH to log THH. It factors through the edge homomorphism of the Tor spectral sequence in the proof of Theorem \ref{thm:main}. The class $\lambda_{n+1}$ is the boundary of $\gamma_p([dx_n])$, while $\mu_{n+1}$ detects the multiplicative extension $\kappa_n^p$. Hence
   \[\mathrm{ker}((\rho_n)_*)=E(\lambda_1,\cdots,\lambda_n)\otimes P(\mu_{n+1})\{\lambda_{n+1}\}\]
   and
   \[\mathrm{im}((\rho_n)_*)=E(\lambda_1,\cdots,\lambda_n)\otimes P(\kappa_n^p).\]
   Set $D_n=E(\lambda_1,\cdots,\lambda_n)\otimes P(\mu_{n+1}).$ Then $\mu_{n+1}$ acts on $V(n)_*\THH(\BP{n},\langle v_n\rangle)$ by multiplication by $\kappa_n^p$ and on $V(n)_*\THH(\BP{n-1})$ by multiplication by $\mu_n^p$, by \cite[Proposition 2.2.2]{AHW24}.
   Consequently, $V(n)_*\THH(\BP{n},\langle v_n\rangle)$ and $V(n)_*\THH(\BP{n-1})$ are free as $D_n$-modules
   \[V(n)_*\THH(\BP{n},\langle v_n\rangle) \cong D_n\{\kappa_n^a,d\mathrm{log}x_n\cdot\kappa_n^a \ | \ 0\leq a<p\}\]
   and
   \[V(n)_*\THH(\BP{n-1}) \cong D_n\{\mu_n^a,\epsilon_n\cdot\mu_n^a \ | \ 0\leq a<p\}.\]
   A degree argument then implies that
   \[(\partial_n)_*(d\mathrm{log}x_n\cdot\kappa_n^k)\doteq\mu_n^k,\]
   by the exactness of the sequence. Moreover, if $p\mid k$, then $\kappa_n^k$ lies in the image of $(\rho_n)_*$, and hence $(\partial_n)_*(\kappa_n^k)=0$. If $p\nmid k$, then $(\partial_n)_*(\kappa_n^k)$ and $(\partial_n)_*(\lambda_n d\mathrm{log}x_n\cdot\kappa_n^{k-1})\doteq\lambda_n\mu_n^{k-1}$ linearly independent. Since $V(n)_{2kp^n-1}\THH(\BP{n-1})$ is generated by $\epsilon_n\mu_n^{k-1}$ and $\lambda_n\mu_n^{k-1}$, we get the desired formula. It follows that
   \[\mathrm{im}((\partial_n)_*)=D_n\{\mu_n^a \ | \ 0\leq a<p\}\oplus D_n\{\epsilon_n\mu_n^a \ | \ 0\leq a<p-1\},\]
   and $(\tau_n)_*$ vanishes on this submodule. The only remaining $D_n$-module generator is $\epsilon_n\mu_n^{p-1}$, and exactness implies
   \[(\tau_n)_*(\epsilon_n\mu_n^{p-1})\doteq\lambda_{n+1}.\]
   The remaining transfer formulas follows from the linearity.
\end{proof}

\subsection{The height-two case} 
We now consider the cofiber sequence \eqref{seq2}. Again, we denote $\overline{\mathrm{res}}_2$ by $\bar{\partial}_2$ and the fiber map of $\bar{\rho}_2$ by $\bar{\tau}_2$. Moreover, set $\overline{D}_2=E(\lambda_1,\epsilon_1)\otimes P(\mu_3)\otimes P_p([dx_1]).$ Then the calculations in Section \ref{sec5} implies that there are isomorphisms of free $\overline{D}_2$-modules
\begin{equation}\label{eq:free}
	\begin{split}
		& V(2)_*\THH(\BP{2}/v_1)\cong \overline{D}_2\{1, \lambda_3\},\\
		& V(2)_*\THH(\BP{2}/v_1, \langle v_2\rangle)\cong \overline{D}_2\{[dx_2]^a, d\log x_2 \cdot [dx_2]^a \ |\ 0\leq a<p\},\\
		& V(2)_*\THH(\mathbb{Z}_{(p)})\cong \overline{D}_2\{\mu_1^{ap}, \epsilon_2\cdot\mu_1^{ap} \ | \ 0\leq a<p\},
	\end{split}
\end{equation}

	where $[dx_1]$ and $\mu_3$ correspond to $\mu_1$ and $\mu_1^{p^2}$ in $V(2)_*\THH(\mathbb{Z}_{(p)})$, respectively.  

\begin{prop}\label{prop:seq2}
	Let $p\geq7$. With the generators chosen above and $d\in\overline{D}_2$, the localization map satisfies
	\[(\bar{\rho}_2)_*(d)=d,\qquad (\bar{\rho}_2)_*(d\cdot \lambda_3)=0.\]
	The residue map satisfies
	\begin{alignat*}{2}
	(\bar{\partial}_2)_*(d) &=0,\\
	(\bar{\partial}_2)_*(d\cdot [dx_2]^a) & \doteq d\cdot \epsilon_2\cdot \ \mu_1^{p(a-1)}, &\qquad& 1\leq a<p,\\
	(\bar{\partial}_2)_*(d\cdot d\log x_2\cdot [dx_2]^a) &\doteq d\cdot \mu_1^{pa}, && 0\leq a<p.
	\end{alignat*}
	The transfer map satisfies
	\begin{alignat*}{2}
	(\bar{\tau}_2)_*(d\cdot \mu_1^{pa}) &=0, && 0\leq a<p,\\
	(\bar{\tau}_2)_*(d\cdot\epsilon_2\cdot\mu_1^{pa})&=0, && 0\leq a<p-1,\\
	(\bar{\tau}_2)_*(d\cdot\epsilon_2\cdot\mu_1^{p(p-1)}) &\doteq d\cdot\lambda_3.
	\end{alignat*}
\end{prop}
\begin{proof}
	The cofiber sequence \eqref{seq2} is obtained from the cofiber sequence \eqref{seq1} for $n=2$, by tensoring with $\THH(\BP{2}/v_1)$ over $\THH(\BP{2} $. Since $\BP{2}$ is an $\Ee_3$-ring spectrum, $\THH(\Ss[x_1])$-action and $\THH(\Ss[x_2])$-action on $\THH(\BP{2})$ are compatible with each other. Hence there is a commutative diagram
	\[\xymatrix{
	\THH(\BP{1}) \ar[r]^-{\tau_2} \ar[d] & \THH(\BP{2}) \ar[r]^-{\rho_2} \ar[d]
	& \THH(\BP{2},\langle v_2\rangle) \ar[d] &\\
	\THH(\Zz_{(p)}) \ar[r]^-{\bar{\tau}_2} & \THH(\BP{2}/v_1) \ar[r]^-{\bar{\rho}_2} & \THH(\BP{2}/v_1,\langle v_2\rangle).
	}\]
	Again, the map $(\bar{\rho}_2)_*$ is determined by the canonical map from ordinary THH to log THH. In particular, we have
	\[\mathrm{ker}(\bar{\rho}_2)_*=\overline{D}_2\{\lambda_3\},\qquad \mathrm{im}(\bar{\rho}_2)_*=\overline{D}_2.\]
	Moreover, since $\lambda_2\in V(2)_{2p^2-1}\THH(\BP{1})$ maps to zero in $ V(2)_{2p^2-1}\THH(\Zz_{(p)})$ by \cite[Proposition 2.2.2]{AHW24}, the naturality implies that 
	\[(\bar{\partial}_2)_*(d\cdot [dx_2]^a)\doteq d\cdot \epsilon_2\cdot \ \mu_1^{p(a-1)}, \qquad (\bar{\partial}_2)_*(d\cdot d\log x_2\cdot [dx_2]^a)\doteq d\cdot \mu_1^{pa},\]
	for $1\leq a<p$ and $0\leq a<p$, respectively. It follows from Equation \eqref{eq:free} that $(\bar{\partial}_2)_*(d)=0$, since $\mathrm{im}(\bar{\rho}_2)_*=\overline{D}_2$. Naturality also implies that 
	\[(\bar{\tau}_2)_*(d\cdot\epsilon_2\mu_1^{p(p-1)}) \doteq d\cdot\lambda_3,\] 
	by Proposition \ref{prop:seq1}. Finally, exactness gives the asserted zero maps.
\end{proof}
\begin{rem}
	Note that since $\BP{2}/v_1$ is only an $\Ee_1$-ring, maps in Proposition \ref{prop:seq2} are not necessarily $\overline{D}_2$-linear.
\end{rem}

\section*{Acknowledgements}
The author is grateful to his advisor Xiangjun Wang for his guidance and support.

\bibliographystyle{plain}
\bibliography{ref}
	
\end{document}